\RequirePackage{fix-cm}
\documentclass[12pt,a4paper]{article}

\usepackage[a4paper]{geometry}
\usepackage{fancyhdr}

\fancypagestyle{plain}{%
	\lhead{\sf \textbf{The final publication is available at:} \
	\href{http://link.springer.com/article/10.1007\%2Fs00211-011-0429-5}{link.springer.com} \\
	Numerische Mathematik (2012) 121:6597 \\
	DOI \href{https://doi.org/10.1007/s00211-011-0429-5}{10.1007/s00211-011-0429-5}
}
	\chead{}
	\rhead{}
}

\usepackage{latexsym}
\usepackage{amsmath, amssymb, amsthm}
\usepackage{mathptmx}
\DeclareMathAlphabet{\mathcal}{OMS}{cmsy}{m}{n}
\DeclareSymbolFont{largesymbols}{OMX}{cmex}{m}{n}

\usepackage{subfigure, graphicx}
\usepackage{bm}
\usepackage[latin1]{inputenc}
\usepackage{hyperref}
\hypersetup{pdfpagemode=UseNone}

\theoremstyle{plain}% default
\newtheorem{theorem}{Theorem}[section]
\newtheorem{lemma}[theorem]{Lemma}

\theoremstyle{definition}

\theoremstyle{remark}
\newtheorem*{remark}{Remark}

\numberwithin{equation}{section}

\newcommand{\CP}{{\mathcal C}{\mathcal P}}

\newcommand{\bfI}{{\boldsymbol I}}

\newcommand{\mcE}{\mathcal{E}}
\newcommand{\mcF}{\mathcal{F}}

\newcommand{\mcP}{\mathcal{P}}
\newcommand{\mcN}{\mathcal{N}}
\newcommand{\mcR}{\mathcal{R}}
\newcommand{\mcC}{\mathcal{C}}
\newcommand{\mcD}{\mathcal{D}}
\newcommand{\mcDPN}{\mathcal{DPV}}
\newcommand{\mcDPM}{\mathcal{DPM}}

\newcommand{\tn}{|\mspace{-1mu}|\mspace{-1mu}|}
\newcommand{\mcK}{\mathcal{K}}
\newcommand{\mcS}{\mathcal{S}}
\newcommand{\Nod}{\mathcal{V}}

\title{
  Continuous Piecewise Linear Finite Elements for the Kirchhoff--Love Plate Equation
}

\author{
Karl Larsson \footnotemark[2]
    \qquad
  Mats G. Larson \footnotemark[3]
\\[4mm]\it Department of Mathematics and Mathematical Statistics, \\\it Ume{\aa} University, 
SE-901 87 Ume{\aa}, Sweden
}

\date{}
\begin{document}

\maketitle
\renewcommand{\thefootnote}{\fnsymbol{footnote}}
\footnotetext[2]{{\tt karl.larsson@math.umu.se}}
\footnotetext[3]{{\tt mats.larson@math.umu.se}}

\begin{abstract}
A family of continuous piecewise linear finite elements for thin plate problems is presented. We use standard linear interpolation of the deflection field to reconstruct a discontinuous piecewise quadratic deflection field. This allows us to use discontinuous Galerkin methods for the Kirchhoff--Love plate equation. Three example reconstructions of quadratic functions from linear interpolation triangles are presented: a reconstruction using Morley basis functions, a fully quadratic reconstruction, and a more general least squares approach to a fully quadratic reconstruction. The Morley reconstruction is shown to be equivalent to the Basic Plate Triangle. Given a condition on the reconstruction operator, a priori error estimates are proved in energy norm and $L^2$ norm. Numerical results indicate that the Morley reconstruction/Basic Plate Triangle does not converge on unstructured meshes while the fully quadratic reconstruction show optimal convergence.
%\subclass{74S05} % Mechanics of deformable solids > Numerical methods > Finite element methods
\end{abstract}

\tableofcontents

\section{Introduction}
The Kirchhoff-Love plate equation is a fourth order partial differential equation modeling the deflection of thin plates. To approximate solutions to this equation using standard finite element methods $C^1$ finite element spaces are required. The difficulty of creating such spaces on unstructured triangulations is a well known problem. A possible $C^1$ element is the conforming Argyris triangle \cite{Argyris1968} which use a fifth order polynomial approximation. Nonconforming options include the Morley triangle \cite{Morley1968} and more recently discontinuous Galerkin (dG) methods \cite{Engel2002,HansboLarson2002}. While it is clear that higher order elements are in many ways superior for modeling the plate equation, an advantage of low order elements lies in modeling complex domains using few degrees of freedom. With the extension to shells and the desired conformity when combining shells and volumes the advantages of low order elements that only feature displacement degrees of freedom become obvious. While this is a possibility when using dG methods, current formulations \cite{Engel2002,HansboLarson2002} require at least piecewise quadratic polynomials to yield accurate results. The focus of this paper is accurate modeling of the plate equation using a continuous piecewise linear deflection field.

Several authors have tried to develop finite element methods for thin plate modeling using a continuous piecewise linear deflection field. Since most terms in the variational formulation then vanish there is a need to discretely approximate higher order quantities to retain sufficient information. Therefore a common trait for this class of elements is that patches of elements are used for these approximations. Nay and Utku \cite{NayUtku1972} used a patch of elements to reconstruct a quadratic deflection field on each element using least squares approximation. Barnes \cite{Barnes1977} introduced a facet triangular plate element where the normal curvature to each edge is approximated from the change in normal gradient to neighboring elements. In a similar approach Hampshire \cite{Hampshire1992} derived a plate element where the stiffness was represented by torsional springs at each edge. Also based on the idea of torsional springs at element edges Phaal and Calladine \cite{PhaalCalladine1992,PhaalCalladine1992-2} presented a family of facet plate and shell elements which use quadratic polynomial reconstruction to calibrate the spring coefficients. By using a mixed interpolation technique in combination with finite volume concepts Oñate and Cervera \cite{OnateCervera1993} and Oñate and Zárate \cite{OnateZarate2000} proposed a procedure for deriving linear thin plate and shell elements.

In this paper we present a framework for constructing continuous piecewise linear finite elements for the Kirchhoff-Love plate equation. The fundamental idea is to use patches of a continuous piecewise linear function to reconstruct a discontinuous piecewise quadratic function which is used in a dG formulation. We apply the framework for reconstructions in a finite element formalism presented in \cite{BarthLarson2002} to a general dG method for the Kirchhoff-Love plate equation \cite{HansboLarson2002}. Three example reconstructions are presented and related to existing elements. Given a condition on the reconstruction operator we prove a priori error estimates in the energy norm and in the $L^2$ norm.

The remainder of this paper is organized as follows; in Section 2 we present the Kirchhoff-Love plate model and the discontinuous Galerkin method using piecewise quadratics continuous at the nodes, in Section 3 we present three reconstructions from continuous piecewise linears into piecewise quadratics, in Section 4 we prove a priori error estimates, and in Section 5 we present convergence studies and numerical examples.

%%%%%%%%%%%%%%%%%%%%%%%%%%%%%%%%%%%%%%%%%%%%%%%%%%%%%%%%%%%%%%%%%%%%%%%%%%%%%%%

\section{The Plate Model and dG Method}

\subsection{The Kirchhoff-Love Plate Model}
The Kirchhoff-Love equilibrium equation governing the deflection of a thin elastic plate occupying a plane domain $\Omega$ takes the form: Given $f$, find the deflection $u$ such that
\begin{equation}\label{eq:equilibrium}
\sigma_{ij,ij} = f \quad \text{in $\Omega$}
\end{equation} 
where we use the summation convention and the comma sign indicates differentiation. The relationship between moments $\sigma_{ij}$ and curvatures $\kappa_{ij}$ is given by
\begin{align}\label{eq:constitutive}
\sigma_{ij} = \lambda \Delta u \delta_{ij} 
                 + \mu \kappa_{ij}(u), \quad i,j=1,2
\end{align} 
where $\delta_{ij}$ is the Kronecker delta, $\Delta$ is the Laplacian, $\lambda$ and $\mu$ are Lamé parameters, and $\kappa_{ij}$ are curvatures defined by $\kappa_{ij}(u) = u_{,ij}$. Using Poisson's ratio $\nu$ and bending stiffness $D$ we can write the Lamé parameters $\lambda =D\nu$ and $\mu = D(1-\nu)$. The bending stiffness of the plate is defined by
\begin{align}
D=\frac{E p^3}{12(1-\nu^2)}
\end{align}
where $E$ is Young's modulus and $p$ is the thickness of the plate.

Let $n=(n_1,n_2)$ be an outwards unit normal to the boundary $\Gamma=\partial\Omega$ and let $t=(t_1,t_2)=(n_2,-n_1)$ be a tangent to $\Gamma$. To define the boundary conditions we need the following quantities
\begin{align}
u_{,n} &= u_{,j}n_j
\\
u_{,t} &= u_{,j}t_j
\\
M_{nn} &= \sigma_{ij} n_i n_j
\\
M_{nt} &= \sigma_{ij} n_i t_j
\\
T &= \sigma_{ij,j} n_i + M_{nt,t} \label{eq:transv}
\end{align}
where $u_{,n}$ and $u_{,t}$ are normal and tangential gradients, $M_{nn}$ and $M_{nt}$ are bending and twisting moments, and $T$ is the transversal force.

We split the boundary into three disjoint parts $\Gamma = \Gamma_C \cup \Gamma_S \cup \Gamma_F$ and let these parts define a clamped boundary, a simply supported boundary, and a free boundary. Let the set of angular corners on $\Gamma_F$ be denoted $\mathcal{X}_F$. The boundary conditions read
\begin{alignat}{2}
&u = u_{,n} = 0 &&\text{on $\Gamma_C$} \label{eq:bc_c}\\
&u = M_{nn} = 0 &&\text{on $\Gamma_S$} \label{eq:bc_s}\\
&M_{nn} = T = 0 &\qquad&\text{on $\Gamma_F$} \label{eq:bc_f} \\
&M_{n^+ t^+} = M_{n^- t^-} &&\text{at $\mathcal{X}_F$} \label{eq:bc_cf}
\end{alignat}
where $\{n^+,t^+\}$ and $\{n^-,t^-\}$ denote the normal and tangent of $\Gamma$ at respective sides of an angular corner.

Let $H^s(\omega)$ denote the Sobolev space of order $s$ on the set $\omega\subset\Omega$, 
with norm $\left\| \cdot \right\|_{s,\omega}$ and semi-norm $\left| \cdot \right|_{m,\omega}$ defined for $m\leq s$.
Introducing the following function space where the essential boundary conditions are imposed
\begin{align}
\mathcal{W} =\{ v \in H^2(\Omega):~\text{$v=v_{,n}=0$ on $\Gamma_C$, $v=0$ on $\Gamma_S$} \}
\end{align}
we recall that the standard variational statement reads: Find $u\in\mathcal{W}$ such that
\begin{align}
(\sigma_{ij}(u),\kappa_{ij}(v)) = (f,v) \quad \text{for all $v\in\mathcal{W}$}
\end{align}
The calculations leading to this variational statement will be performed in Section \ref{sect:varstatment}, albeit on an element level.
%Note that using the constitutive relationship \eqref{eq:constitutive} we can write
%\begin{align}
%(\sigma_{ij}(u),\kappa_{ij}(v)) = \lambda(\Delta u,\Delta v) + \mu(\kappa_{ij}(u),\kappa_{ij}(v))
%\end{align}

\subsubsection{The Mesh and Discontinuous Space} \label{sectionmesh}
Let $\mcK=\{K\}$ be a triangulation of $\Omega$ into geometrically conforming shape regular triangles. We denote the diameter of element $K$ by $h_K$ and the global mesh size parameter by $h=\max_{K\in\mcK}h_K$. 
Further, let the mesh be quasi-uniform such that
\begin{align} \label{eq:quasiu}
c h \leq h_K \leq C h \quad \text{for all $K$}
\end{align}
where $c$ and $C$ are mesh independent constants.
The set of edges in the mesh is denoted by $\mcE=\{E\}$ and the set of nodes in the mesh is denoted by $\Nod=\Nod(\mcK)=\Nod(\mcE)$. We split $\mcE$ into disjoint subsets
\begin{align}
\mcE = \mcE_I \cup \mcE_C\cup \mcE_S\cup \mcE_F
\end{align}
where $\mcE_I$ is the set of edges in the interior of $\Omega$,
$\mcE_C$ is the set of edges on $\Gamma_C$, etc.
Further, with each edge we associate a fixed unit normal $n_E$ and a corresponding unit tangent $t_E$ such that for edges on the boundary $n_E$ is the exterior unit normal. On each node $\partial E$ belonging to edge $E$ we define $\partial n_E = 1$ if $t_E$ points outwards from $E$ and $\partial n_E = -1$ is $t_E$ points inwards to $E$.

For reasons that become evident when we define the reconstruction operators we make a special construction: for every exterior edge $E\in\mcE\backslash\mcE_I$ we add a ghost element outside the domain by placing an additional a degree of freedom, a ghost node, such that the ghost element becomes anti-symmetric to the interior element, see Figure~\ref{fig:boundarypatch}. We denote the set of ghost elements by $\mcK_G$.

Next we define a number of function spaces: Let $\CP_{1}(\mcS)$ denote the space of continuous piecewise linear functions with support on a set of elements $\mcS$
\begin{align}
\CP_{1}(\mcS) =\{ v \in C^0(\Omega):~\text{$v \vert_K \in \mcP_{1}(K)$ for all $K \in \mcS$} \}
\end{align}
and let $\CP_{1}$ denote the space of continuous piecewise linear functions with support on $\mcK \cup \mcK_G$ and zero on the clamped and the simply supported boundary
\begin{align}
\CP_{1} =\{ v \in \CP_{1}(\mcK \cup \mcK_G):~\text{$v=0$ on $x \in \mcE_C \cup \mcE_S$} \}
\end{align}
Furthermore, let $\mcD\mcP_{2}$ denote the space of discontinuous piecewise quadratic polynomials
\begin{align}
\mcD\mcP_{2} =\{ v :~\text{$v \vert_K \in \mcP_{2}(K)$ for all $K \in \mcK$} \}
\end{align}
and finally let $\mcDPN$ denote the space of discontinuous piecewise quadratic polynomials that 
are continuous at the nodes and zero on nodes associated with the clamped and the simply supported boundaries
\begin{align}
\mcDPN = \{ v \in \mcD\mcP_2 :~\text{$v$ continuous in $x \in \Nod$, $v=0$ in $x \in \Nod(\mcE_C \cup \mcE_S)$}  \}
\end{align}

To formulate our method we will use the following notation for the average
\begin{equation}
\langle v \rangle = 
\begin{cases}
(v^+ + v^-)/2 & E \in \mcE_I
\\
v^+ & E \in \mcE \setminus \mcE_I 
\end{cases}
\end{equation}
and for the jump
\begin{equation}
[v] = 
\begin{cases} 
v^+ - v^-  & E  \in \mcE_I
\\
v^+ & E  \in \mcE \setminus \mcE_I
\end{cases}
\end{equation}
of a function $v$ at an edge $E$, where $v^{\pm}=\lim_{\epsilon\rightarrow 0^+}v(x \mp \epsilon n_E)$ with $x\in E$.

\subsubsection{Variational Formulation on an Element} \label{sect:varstatment}
As a motivation for the dG method we will here derive a variational formulation on each element. We multiply \eqref{eq:equilibrium} by a test function $v \in H^4 = H^4(\Omega)$ and integrate over $K$. Applying Green's formula two times gives
\begin{align}  \label{eq:greens0}
\begin{split}
(\sigma_{ij,ij},v)_K  &= -(\sigma_{ij,i},v_{,j})_K 
          + (\sigma_{ij,i},v n_j)_{\partial K}
\\
&= (\sigma_{ij},v_{,ij})_K - (\sigma_{ij} n_i ,v_{,j})_{\partial K}
        + (\sigma_{ij,i},v n_j)_{\partial K}
\\
&=(\sigma_{ij},v_{,ij})_K 
   - (M_{nn},v_{,n})_{\partial K}
 - (M_{nt},v_{,t})_{\partial K}
         + (\sigma_{ij,i}, v n_j)_{\partial K}
\end{split}
\end{align}
where we use that $v_{,j} = v_{,n} n_j + v_{,t} t_j$ in the last equality.

Partial integration along an edge segment $E$ gives
\begin{equation}\label{eq:greens1}
(M_{nt},v_{,t})_{E} = -(M_{nt,t}, v)_E 
           + (M_{nt}, v n_{\partial E} )_{\partial E}
\end{equation}

Combining \eqref{eq:greens0} and \eqref{eq:greens1} we have the following variational formulation on the element level
\begin{align}\label{eq:elem_variational}
( \sigma_{ij} (u), \kappa_{ij} (v) )_K 
- \sum_{E \subset \partial K} \Bigl(
   (M_{nn},v_{,n})_{E}
      - (T ,v)_{E}
         + (M_{nt}, v n_{\partial E})_{\partial E} \Bigr) = (f, v)_K
\end{align}
for all $v\in H^4$.

\subsubsection{Discrete Moments and Corner Forces}
By giving definitions of the bending and twisting moments and the transversal force on element edges for functions in $\mcDPN$ which is consistent for functions in $H^4$ we can extend the elementwise variational statement \eqref{eq:elem_variational} to a variational statement on $H^4 + \mcDPN$.
Following the procedure in \cite{HansboLarson2002} and motivated by the proof of Lemma \ref{lemma2} below we for $v \in H^4 + \mcDPN$ introduce the following definitions of these quantities on each element edge $E\in\mcE$ unless previously defined by boundary conditions:
\begin{alignat}{2}
M_{nn}(v) &= \langle M_{nn}(v) \rangle - \beta h^{-1} P_{0} [ v_{,n} ]  \label{mnn_def} \\ 
T(v) &= \langle T(v) \rangle  \label{defT} \\
M_{nt}(v) &= \langle M_{nt}(v) \rangle \label{defMnt}
\end{alignat}
where $\beta$ is a positive parameter and $P_0$ is the $L^2$ projection onto the space of constants. Using these definitions in \eqref{eq:elem_variational} and summing over all elements $K\in\mcK$ yields a variational statement on $H^4 + \mcDPN$.

%Note that while we give definitions for bending moment \eqref{mnn_def}, twisting moment \eqref{defMnt} and transversal force \eqref{defT} only the bending moment will actually contribute to the resulting method.

%As $\mcDPN$ is a space of piecewise quadratic functions $\langle T(v) \rangle = 0$. Thus in the resulting method terms containing the transversal force will vanish on all interior edges and on the boundary they will also vanish due to the homogenous boundary conditions. However, for consistency these terms remain in the variational statement.

Due to the nodal continuity of $H^4+\mcDPN$ terms containing the twisting moment will vanish on all interior edges. On the boundary pointwise twisting moments will appear where the boundary is not smooth, but given the homogeneous boundary conditions these terms will be zero on $\Gamma_C \cup \Gamma_S$ as $v=0$, and also zero on $\Gamma_F$ due to \eqref{eq:bc_cf}.

The resulting variational statement is nonsymmetric but we may symmetrize the variational statement without affecting consistency as the added terms become zero for the exact solution.

Next we present the resulting variational statement on $H^4 + \mcDPN$.

\subsubsection{Extended Variational Statement}
The extended variational statement reads: Find $u \in H^4 + \mcDPN$ such that
\begin{align} \label{galerkin1}
	a(u,v)=l(v)	\quad \text{for all }v\in H^4 + \mcDPN
\end{align}
where the bilinear form is defined by
\begin{multline} \label{eq:BilinearQuad}
a(v,w) = \sum_{K\in\mathcal{K}} (\sigma_{ij}(v), \kappa_{ij}(w))_K
\\
 - \sum_{E\in\mathcal{E}\setminus(\mathcal{E}_S\cup\mathcal{E}_F)} \Big( \left(\left\langle M_{nn}(v) \right\rangle,\left[w_{,n}\right]\right)_E + \left(\left[v_{,n}\right],\left\langle M_{nn}(w) \right\rangle\right)_E
\\
 - \beta (h^{-1} P_0 [v_{,n}], P_0 [w_{,n}])_E \Big)
\\
 + \sum_{E\in\mathcal{E}\setminus\mathcal{E}_F} \Big( \left(\left\langle T(v) \right\rangle,\left[w\right]\right)_E + \left(\left[v\right],\left\langle T(w)\right\rangle\right)_E \Big)
\end{multline}
where $\beta$ is a real parameter and the linear functional is defined by
\begin{align} \label{eq:linearform}
l(v)=(f,v)
\end{align}

We now move on to formulate the dG method.

\subsection{The dG Method with Piecewise Quadratics Continuous at Nodes}
The dG method for the plate equation with piecewise quadratic functions continuous at the nodes can now be formulated as follows:
Find $U \in \mcDPN$ such that
\begin{align} \label{method1}
	a(U,v)=l(v)	\quad \text{for all }v\in\mcDPN
\end{align}
where the bilinear form is given by \eqref{eq:BilinearQuad} and the linear functional is given by \eqref{eq:linearform}. Note that the last sum in the bilinear form \eqref{eq:BilinearQuad} gives no contribution as $\langle T(v) \rangle=0$ for $v\in\mcDPN$.
The boundary condition $u_{,n}=0$ on $\Gamma_C$ is weakly enforced via the $\beta$ penalty term while the condition $u=0$ on $\Gamma_C \cup \Gamma_S$ is strongly enforced at the nodes.

For a more general dG method for the plate equation without the restriction to nodal continuity and piecewise quadratics in the approximation of the deflection field we refer to \cite{HansboLarson2002}.

\subsection{The dG Method with Embedded Continuous Piecewise Linears}

To formulate our method using a continuous piecewise linear deflection field we use the framework presented in \cite{BarthLarson2002} for using reconstructions in a finite element formalism. We let $\mathcal{R}$ be a reconstruction operator which embeds the space of continuous piecewise 
linear polynomial functions $\mathcal{CP}_1$ into the space $\mathcal{DPV}$ of discontinuous 
piecewise quadratic polynomials continuous at the nodes:
\begin{align}
\mathcal{R}:\mathcal{CP}_1 \hookrightarrow \mathcal{DPV}
\end{align}
Also let the following criterion on the reconstruction operator hold: For $v\in\mathcal{CP}_1$
\begin{align} \label{eq:strongRcont}
v = \mcR v ,\quad \text{for all $x\in\Nod$}
\end{align}
The discontinuous Galerkin method with embedded continuous piecewise linear functions takes the following form: 
Find $U \in\mathcal{CP}_1$ such that
\begin{align}\label{dgrec}
	a(\mathcal{R}U,\mathcal{R}v) = l(\mathcal{R}v), \quad \text{for all $v \in \mathcal{CP}_1$}
\end{align}
where $a(\cdot,\cdot)$ and $l(\cdot)$ are defined in \eqref{eq:BilinearQuad} and \eqref{eq:linearform}. The clamped boundary condition is weakly enforced by the $\beta$ penalty parameter on $\mcR U$. As $U$ coincides with $\mcR U$ at the nodes we can choose to strongly enforce $u=0$ on $\Gamma_C \cup \Gamma_S$ directly on $U$.

%%%%%%%%%%%%%%%%%%%%%%%%%%%%%%%%%%%%%%%%%%%%%%%%%%%%%%%%%%%%%%%%%%%%%%%%%

\section{Examples of Reconstruction Operators}\label{section3}

In this section we consider three reconstruction operators in the presented framework, all of which embed continuous piecewise 
linear functions into $\mcDPN$. To reconstruct a quadratic function on an element $K$ these operators use the vertex information in a patch of elements. In the first example we reconstruct into the space of the quadratic Morley basis functions, which is the subspace of functions in $\mcDPN$ that have continuous 
normal derivative at element edge midpoints. We show that this method is equivalent to the Basic Plate Triangle 
presented in \cite{OnateCervera1993,OnateZarate2000}. The second example is a fully quadratic reconstruction into 
$\mcDPN$ using a four element patch. In the last example reconstruction we handle special cases where the fully quadratic reconstruction breaks down due to the mesh configuration. A least squares approach to fully quadratic reconstruction is used to allow larger patches when fully quadratic reconstruction from a four element patch fails.

\subsection{Patch of Elements} \label{section:patch}
To reconstruct a complete quadratic polynomial six independent degrees of freedom are required. Thus, a patch of continuous piecewise linear elements is needed to represent sufficient information. We denote the patch that is used for reconstructing a quadratic function on element $K$ by $\mcN(K)$ and let it consist of connected elements in a neighborhood of $K$. Let the patch have finite size such that
\begin{align} \label{finitepatch}
\text{diam}(\mcN(K)) \leq C h_K
\end{align}
where $C$ is a mesh independent constant.

In a triangle mesh a patch $\mcN(K)$ typically is the standard four element patch illustrated in Figure~\ref{fig:standardpatch} consisting of $K$ and the three elements neighboring $K$. For elements neighboring the boundary the patch will include a ghost element outside the domain for each element edge belonging to the boundary, see Figure~\ref{fig:boundarypatch}. As defined in Section~\ref{sectionmesh} the locations of the ghost nodes are set such that the ghost elements are anti-symmetric with respect to $K$, thus preserving properties of structured meshes.

\begin{figure}
\centering
 \subfigure[Standard four element patch.]{
\includegraphics[width=4cm]{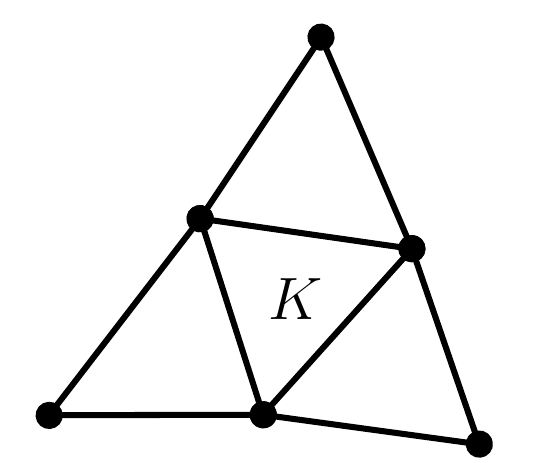}
   %\import{illustrations/}{standardpatch.eps_tex}
	\label{fig:standardpatch}
 }
 \subfigure[Patch on boundary.]{
	\includegraphics[width=4cm]{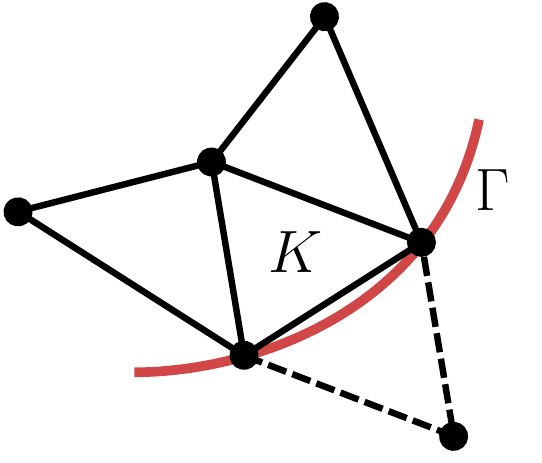}
   %\import{illustrations/}{boundarypatch.eps_tex}
	\label{fig:boundarypatch}
 }
 \caption{Standard patches for (a) interior elements and (b) elements neighboring the boundary.}
\end{figure}

\subsection{Morley Reconstruction}
It is well known that the nonconforming Morley element \cite{Morley1968} shows optimal convergence in the 
approximation of the Kirchhoff-Love plate bending equation. As noted in \cite{HansboLarson2002} this element 
is naturally derived in the setting of dG methods for the plate equation by letting $\beta\rightarrow\infty$ in \eqref{eq:BilinearQuad}. An advantage of 
reconstruction using Morley basis functions is that the jump in the normal derivative at the edge midpoint per 
definition is zero which results in that all interior and exterior edge terms $E\in\mcE\backslash\mcE_C$ disappear in the bilinear form (\ref{eq:BilinearQuad}).

The Morley basis functions are constructed so that the deflection field is continuous at the nodes and the gradient 
in the normal direction is continuous at each edge midpoint $x_E$. Clearly this is a subspace of $\mcDPN$ and as such we define the space of Morley functions
\begin{align}
\mathcal{DPM} = \{ v \in \mcDPN :~\text{$\left.[v_{,n}]\right|_{x_E}=0$ for all $E\in\mathcal{E}\setminus(\mathcal{E}_S\cup\mathcal{E}_F)$}  \}
\end{align}

We define the reconstruction of the normal gradient at 
an element edge to be the average normal gradient of the two neighboring linear triangles. Let $\Nod(K)$ be the set of nodes for element $K$ and let $\Nod_E(K)$ be the set of edge midpoints for element $K$. The reconstruction operator $\mathcal{R}:\mcC\mcP_1 \hookrightarrow \mcDPM$ is defined by $(\mcR u) |_K = (\mcR_K u) |_K$ where $\mcR_K: \mcC\mcP_1(\mcN(K)) \rightarrow \mcP_2(\mcN(K))$ is defined as follows
\begin{align} \label{eq:MorleyRecon}
\mcR_K:
\left\{\begin{aligned}
\mcR_K v &= v \ ,&& x\in\Nod(K) \\
(\mcR_K v)_{,n} &= \left\langle v_{,n}\right\rangle \ , &\qquad& x\in\Nod_E(K)
\end{aligned}\right.
\end{align}
Next, we will show that this choice of reconstruction yields a method equivalent to the Basic Plate Triangle.

\subsubsection{Equivalence with Basic Plate Triangle} \label{sect:BPTequiv}
%{Put in short description of their element here. Make sure this section can be read 
% without refering to the original reference}
The Basic Plate Triangle (BPT) presented in \cite{OnateCervera1993,OnateZarate2000} is a triangular plate element using continuous piecewise linear deflections and is derived by combining finite element and finite volume techniques. We will now describe our interpretation for derivation of the BPT in the presented setting, whereafter we will show equivalence with the method produced by the above choice of Morley reconstruction.

The BPT is a mixed interpolation method where the curvatures and moments are approximated using piecewise constant functions and the deflection field is approximated using functions in $\CP_1$. The fundamental idea in this derivation is that by using partial integration of the curvatures such that
\begin{align} \label{eq:PIcurvatures}
(\kappa_{ij}(u),1)_K = (u_{,ij},1)_K = \left( u_{,i}, n_j\right)_{\partial K} , \qquad i,j=1,2
\end{align} 
and equivalently for the moments
\begin{align} \label{eq:PImoments}
(\sigma_{ij}(u),1)_K = \left(\lambda u_{,n} \delta_{ij} 
                 + \mu u_{,i} n_j , 1 \right)_{\partial K} , \qquad i,j=1,2
\end{align}
these terms can be estimated using a $\mathcal{CP}_1$ deflection field.

Starting with the element contribution to the bilinear form \eqref{eq:BilinearQuad} on each element we have
\begin{align} \label{eq:KContrib}
a_K(u,v) = (\sigma_{ij}(u),\kappa_{ij}(v))_K
\end{align}
By using that the curvatures and moments are assumed constant on the element and applying \eqref{eq:PIcurvatures} and \eqref{eq:PImoments} we get
\begin{align} \label{eq:KContrib2}
a_K(u,v) = \frac{1}{\left|K\right|} (\sigma_{ij}(u),1)_K (\kappa_{ij}(v),1)_K
				 = \frac{1}{\left|K\right|} \left( \lambda u_{,n} \delta_{ij} + \mu u_{,i} n_j ,1 \right)_{\partial K}\left( v_{,i} , n_j \right)_{\partial K}
\end{align}
A deflection field $U\in\mathcal{CP}_1$ is then assumed. As the gradient of a continuous piecewise linear function is undefined on element edges they are defined as the average gradient of neighboring elements
\begin{align}
	U_{,i}|_E \equiv \left\langle U_{,i} \right\rangle , \qquad i=1,2
\end{align}
and likewise for the gradient of the test function $v \in \CP_1$.
Note that this definition of the gradient on edges makes all edge terms from the bilinear form \eqref{eq:BilinearQuad} to be zero in the method, except for edges $E\in\mcE_C$, i.e. the clamped boundary. In the derivation of the BPT the normal gradients naturally appear due to the partial integration and are thus enforced weakly on the clamped boundary. Thus, there is no need to extend patches on clamped edges with ghost elements but if we would the boundary condition would read
\begin{align} \label{BPTclamped}
\left\langle U_{,n} \right\rangle = 0 , \quad \text{on $E \in \mcE_C$}
\end{align}

The BPT method is formulated as follows: Find $U\in\mathcal{CP}_1$ such that
\begin{align} \label{eq:MethodBPT}
\sum_{K\in\mcK}{\hat{a}_K(U,v)} = l(v) , \quad \text{for all $v\in\mathcal{CP}_1$}
\end{align}
where
\begin{align}
\hat{a}_K(U,v) &= \frac{1}{\left|K\right|} \left( \lambda \left\langle U_{,n} \right\rangle \delta_{ij} + \mu \left\langle U_{,i} \right\rangle n_j , 1 \right)_{\partial K}
																		\left( \left\langle v_{,i} \right\rangle , n_j \right)_{\partial K}
\end{align}
The average gradient of $U$ is constant on each edge which means the integrals are exactly evaluated by midpoint quadrature. We get
\begin{align} \label{eq:BiBPT}
\hat{a}_K(U,v) =
	\frac{1}{\left|K\right|} \left( \sum_{E\in\partial K} h_E \bigl(\lambda \left\langle U_{,n} \right\rangle \delta_{ij} + \mu \left\langle U_{,i} \right\rangle n_j \bigr) \big|_{x_E} \right)
	\left( \sum_{E\in\partial K} h_E \bigl(\left\langle v_{,i} \right\rangle n_j \bigr) \big|_{x_E} \right)
\end{align}
where $x_E$ is the midpoint of each edge.

We will now show that the proposed method (\ref{dgrec}) when using the above Morley reconstruction is equivalent to the BPT (\ref{eq:MethodBPT}, \ref{eq:BiBPT}).
%We now turn to showing equivalence with our method when choosing Morley reconstruction.
As previously noted reconstructions into Morley space give no edge terms in the bilinear form \eqref{eq:BilinearQuad}, except for the clamped boundary, so the finite element method reads: Find $U\in\mathcal{CP}_1$ such that
\begin{align}
\sum_{K\in\mcK}{a_K(\mcR_K U,\mcR_K v)} + \sum_{E\in\mcE_C}{a_E(\mcR_K U,\mcR_K v)} = l(\mcR_K v) , \quad \text{for all $v\in\mathcal{CP}_1$}
\end{align}
where $a_K( \cdot , \cdot )$ is defined in \eqref{eq:KContrib} and $a_E( \cdot , \cdot )$ can be identified in the bilinear form \eqref{eq:BilinearQuad}. This boundary term allows us to enforce clamped boundary conditions weakly. As $(\mcR U)_{,n} = \left\langle U_{,n}\right\rangle$ in the Morley reconstruction the enforcement of the clamped boundary condition is equivalent to \eqref{BPTclamped} for large enough $\beta$.

Apart from the difference in how clamped boundary conditions are enforced, there is also a difference in how the load is calculated in the two methods. Disregarding this difference for now, if we can show that $\hat{a}_K(U,v) = a_K(\mcR_K U,\mcR_K v)$ for our choice of $\mcR_K$, the Morley reconstruction yields a method equivalent to the BPT. As the reconstructed functions in the above equation are quadratic, both curvatures $\kappa_{ij}$ and moments $\sigma_{ij}$ are constant. Thus, we may apply the calculations of \eqref{eq:KContrib2} and yield
\begin{align}
a_K(\mcR_K U,\mcR_K v) 
	= \frac{1}{\left|K\right|} \left( \lambda (\mcR_K U)_{,n} \delta_{ij} + \mu (\mcR_K U)_{,i} n_j , 1 \right)_{\partial K}
																		\left( (\mcR_K v)_{,i} , n_j \right)_{\partial K}
\end{align}
As the gradient $(\mcR_K U)_{,i}, \ i=1,2$ is a linear function, the integrals in the expression above are also exactly evaluated through midpoint quadrature. Thus, we have
\begin{multline} \label{eq:BiMorley}
a_K(\mcR_K U,\mcR_K v) =
	\frac{1}{\left|K\right|} \left( \sum_{E\in\partial K} h_E \bigl(\lambda (\mcR_K U)_{,n} \delta_{ij} + \mu (\mcR_K U)_{,i} n_j \bigr) \big|_{x_E} \right)
	\times \\
	\left( \sum_{E\in\partial K} h_E \bigl( (\mcR_K v)_{,i} n_j \bigr) \big|_{x_E} \right)
\end{multline}
where $x_E$ is the midpoint of each edge. Comparing \eqref{eq:BiBPT} with \eqref{eq:BiMorley} we see that the methods are equivalent if $(\mcR_K w)_{,i}|_{x_E}=\left\langle w_{,i} \right\rangle|_{x_E}, \ i=1,2$ for $w\in\mathcal{CP}_1$. Looking at the normal component of the gradient we have
\begin{align}
(\mcR_K w)_{,n}|_{x_E} = \left\langle w_{,n} \right\rangle|_{x_E}
\end{align}
by definition of the reconstruction operator $\mcR_K$. As the reconstructed function $\mcR_K w$ is quadratic and equal to $w$ at the triangle nodes we know that the derivative of $\mcR_K w$ at a midpoint $x_E$ in the tangential direction is equal to the derivative in the tangential direction of the plane defined by the triangle nodes. Using that $w_{,t}$ is continuous over element edges we have
\begin{align}
(\mcR_K w)_{,t}|_{x_E} = w_{,t}|_{x_E} = \left\langle w_{,t} \right\rangle|_{x_E}
\end{align}
Thus $(\mcR_K w)_{,i}|_{x_E}=\left\langle w_{,i} \right\rangle|_{x_E}, \ i=1,2$ which means that the Morley reconstruction yields a method equivalent with BPT, apart from the mentioned differences in enforcement of clamped boundary conditions and in load calculation.

\subsection{Fully Quadratic Reconstruction} \label{section:FullyQuadRecon}

For this reconstruction operator we consider for each triangle $K$ the 
neighborhood $\mcN(K)$ of triangles that share an edge with $K$. Let $\Nod(\mcN(K))$ be the set of nodes in $\mcN(K)$. Then we define $(\mcR u) |_K = (\mcR_K u) |_K$ where $\mcR_K: \mcC\mcP_1(\mcN(K)) \rightarrow \mcP_2(\mcN(K))$ is defined as follows
\begin{align}\label{eq:QuadraticRecon}
\mcR_K:
(\mcR_K v) (x) = v(x) \ , \qquad   x \in \Nod(\mcN(K))
\end{align}
In general, except for some special configurations of the nodes in $\mcN(K)$, this is a well posed 
problem.
%\subsubsection{Alternative}
%For this reconstruction operator we consider for each triangle $K$ the 
%neighborhood $\mcN(K)$ of triangles that share an edge with $K$. Let $\Nod(\mcN(K))$ be the set of vertices in $\mcN(K)$ and introduce a numbering such that $\{x_i\}=\Nod(\mcN(K))$. We define basis functions $\{\phi_i\}\in\mcP_2$ by
%\begin{align}
	%&\phi_j(x_i) =\delta_{ij} &, i,j=1..6&
%\end{align}
%Then we define $(\mcR u) |_K = (\mcR_K u) |_K$ where $\mcR_K: \mcC\mcP_1(\mcN(K)) \rightarrow \mcP_2(\mcN(K))$ 
%as follows
%\begin{align}\label{eq:QuadraticRecon}
%\mcR_K: \left\{\begin{aligned}
%(\mcR_K v) (x) &= \sum_{i=1}^{6}{\beta_i\phi_i(x)} \\
%\beta_i &= v(x_i)& \quad&, i=1..6
%\end{aligned}\right.
%\end{align}
%In general, except for some special configurations of the nodes in $\mcN(K)$, this is a wellposed 
%problem.

\subsubsection{Relation to Morley Reconstruction} \label{sect:StructMesh}
Consider the notation in Figure~\ref{fig:struct_mesh_criteria}. We define a structured mesh to be a mesh where the midpoint between $x_b$ and $x_d$ will be $x_E$, a criterion which we may formulate as
\begin{align} \label{eq:StructCriterion}
x_E = \frac{x_a+x_c}{2} = \frac{x_b+x_d}{2}
\end{align}
A quadratic function with known values at $x_a$ and $x_c$ will at the midpoint $x_E$ have a tangential gradient equal to the slope of a linear function with the same known values at $x_a$ and $x_c$. As this is valid for the quadratic polynomials associated with both $K^+$ and $K^-$ the jump in the tangential gradient for these polynomials is zero at the midpoint. 
The same reasoning is true for the points $x_b$ and $x_d$, and since the midpoint is the same on structured meshes, we conclude that the jump in the gradient is zero at $x_E$. Thus, for a structured mesh all interior edge terms disappear in \eqref{eq:BilinearQuad} as midpoint quadrature exactly evaluates these terms. 
In this case the fully quadratic reconstruction is identical to the Morley reconstruction as the gradient 
at $x_E$ is continuous in both cases. Given a structured mesh, any theoretical results based on the fully quadratic 
reconstruction is thus applicable to the Morley reconstruction/BPT-element.

\begin{figure}
  \centering
	\includegraphics[width=5cm]{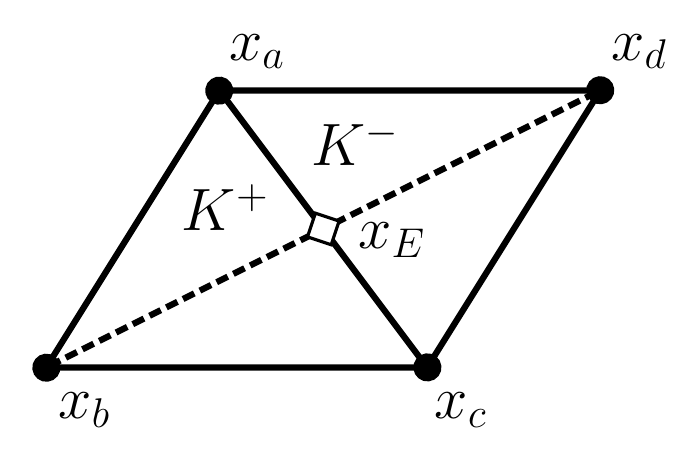}
%\import{illustrations/}{struct_mesh_criteria.eps_tex}
\caption{Illustration of two neighboring elements on a structured mesh.} \label{fig:struct_mesh_criteria}
\end{figure}

\subsubsection{Degenerate Patch Configurations}
While it is unlikely that quality mesh generation will produce patch configurations where the fully quadratic reconstruction fails, we have identified two possible configurations of the standard patch where the fully quadratic reconstruction does fail. We call these degenerate patch configurations.

It is possible that two elements neighboring $K$ share two nodes as illustrated in Figure~\ref{fig:fivenodepatch}, and thus only have five degrees of freedom. Obviously this is insufficient for reconstructing a complete quadratic polynomial.

The other degenerate patch configuration occurs when the set of nodes in the patch includes four nodes positioned on the same straight line as illustrated in Figure~\ref{fig:fournodesaligned}. Along any straight line the quadratic polynomial reduces to a one dimensional quadratic polynomial which is fully described using only three nodal values.

In the next section we will suggest a reconstruction operator that allow extending the patch in the case of a degenerate configuration of the nodes.

\begin{figure}
\centering
 \subfigure[]{
\includegraphics[width=4cm]{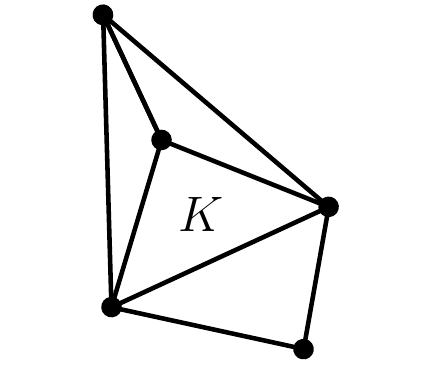}
%   \import{illustrations/}{degenerate1.eps_tex}
\label{fig:fivenodepatch}
 }
 \subfigure[]{
\includegraphics[width=4cm]{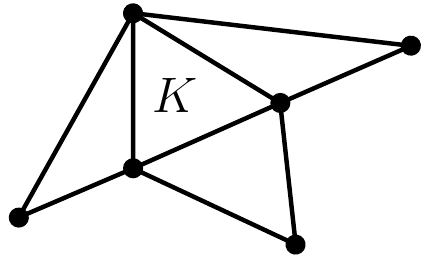}
 % \import{illustrations/}{degenerate2.eps_tex}
\label{fig:fournodesaligned}
 }
 \caption{Degenerate configurations of the standard patch. (a) Standard patch only containing five nodes. (b) Standard patch where four nodes are positioned along a straight line.}
\end{figure}

%\begin{figure}
  %\centering
%\import{illustrations/}{degenerate1.eps_tex}
%\caption{Degenerate configuration of the standard patch only containing five nodes.}
%\label{fig:fivenodepatch}
%\end{figure}
%
%\begin{figure}
  %\centering
%\import{illustrations/}{degenerate2.eps_tex}
%\caption{Degenerate configuration of the standard patch where four nodes are positioned along a straight line.}
%\label{fig:fournodesaligned}
%\end{figure}

\subsection{Least Squares Fully Quadratic Reconstruction}

To deal with the degenerated cases we consider a larger patch of elements in a neighborhood of $K$
and define the reconstruction by exact fitting at the nodes of $K$ and least squares 
fitting at the remaining nodes in the patch.
Let $\Nod(\mcS)$ be the set of nodes in a set of elements $\mcS$.
%Given a linear function $w: K\rightarrow \mathbb{R}$ we describe the space $\mcP_2^w$ as
%\begin{align}
%\mcP_2^w=\{ v: v\in\mcP_2, \text{ and } v(x)=w(x) \text{ for all } x\in\Nod(K) \}
%\end{align}
%which is a subspace to $\mcP_2$.
Again we define $(\mcR u) |_K = (\mcR_K u) |_K$ where $\mcR_K: \mcC\mcP_1(\mcN(K)) \rightarrow \mcP_2(\mcN(K))$ 
is defined as follows
%\begin{align}\label{eq:QuadraticRecon}
%&\mathcal{R}_K u = u \quad \text{at each node in element $K$}
%\\
%&\text{least squares sum minimized over the rest of the nodes}
%\end{align}
\begin{align}\label{eq:LSRecon}
\mcR_K: \left\{\begin{aligned}
&(\mcR_K v)(x) = v(x) \ , &\quad & x\in\Nod(K) \\
&\min_{\mcR_K} \sum_{x \in \Nod_{N} }{\left( (\mcR_K v) (x) - v(x) \right)^2} \ , &\quad & \Nod_{N} = \Nod(\mcN(K))\backslash\Nod(K)
\end{aligned}\right.
\end{align}

The patch of elements $\mcN(K)$ is in general the four element standard patch and the above reconstruction is then identical to the fully quadratic reconstruction. However, if a degenerate patch is detected we extend $\mcN(K)$ one element at a time using elements neighboring $\mcN(K)$ until the patch is no longer degenerate.
%\subsubsection{Alternative}
%For this reconstruction operator we consider for each triangle $K$ the 
%neighborhood $\mcN(K)$ of triangles that share an edge with $K$. Let $\Nod(\mcN(K))$ be the set of vertices in $\mcN(K)$. We introduce a numbered set $\{\hat{x}_i\}$ containing $\Nod(K)$ and the edge midpoints of $K$. We define the standard quadratic basis functions on the triangle $\{\phi_i\}\in\mcP_2$ by
%\begin{align}
	%&\phi_j(\hat{x}_i) =\delta_{ij} &, i,j=1..6&
%\end{align}
%We introduce the numbered set $\{x_i\}=\Nod(\mcN(K))$ where $\{x_i\}_1^3=\Nod(K)$. Then we define $(\mcR u) |_K = (\mcR_K u) |_K$ where $\mcR_K: \mcC\mcP_1(\mcN(K)) \rightarrow \mcP_2(\mcN(K))$ 
%as follows
%\begin{align}\label{eq:LSReconAlt}
%\mcR_K: \left\{\begin{aligned}
%(\mcR_K v) (x) = \sum_{i=1}^{6}{\beta_i\phi_i(x)} \\
%\beta_i = v(x_i) \quad, i=1,2,3 \\
%\min_{\{ \beta_4,\beta_5,\beta_6 \}} \sum_{i=4}^{6}{\left( (\mcR_K v) (x_i) - v(x_i) \right)^2}
%\end{aligned}\right.
%\end{align}

%%%%%%%%%%%%%%%%%%%%%%%%%%%%%%%%%%%%%%%%%%%%%%%%%%%%%%%%%%%%%%%%%%%%%%%%%

\section{A Priori Error Estimates}

 We equip $H^4 + \mcDPN$ with the following energy norm
\begin{multline} \label{eq:energynorm}
  \tn v \tn^2 = \sum_{K \in \mcK} 
 ( \sigma_{ij}(v), \kappa_{ij}(v) )_K
             + h \| \langle M_{nn}(v) \rangle \|_{\partial K  \setminus (\mcE_F \cup \mcE_S)}^2
												 + h^3 \| \langle T(v) \rangle \|_{\partial K  \setminus \mcE_F}^2 \\
                    +  h^{-1} \| P_0 [v_{,n}]\|_{\partial K \setminus (\mcE_F \cup \mcE_S)}^2
\end{multline}  
We note that $\tn \cdot \tn $ is indeed a norm on $H^4 + \mcDPN$ since if
$\sum_{K \in \mcK}( \sigma_{ij}(v), \kappa_{ij}(v) )_K = 0$ then $v$ must
be a piecewise linear function which due to nodal continuity also is continuous.
If also $\sum_{K \in \mcK} \| P_0 [v_{,n}]\|_{\partial K \setminus (\mcE_F \cup \mcE_S)}^2=0$ then $v$ is globally linear.
Finally, for a well posed problem we either need $\Gamma_C \neq \emptyset$ or that there exists no single straight line $\Gamma_\text{line}$ such that
$\Gamma_S \subset \Gamma_\text{line}$. In either case we get $v=0$.

Before turning to our main a priori error estimate we formulate a few lemmas 
that will be needed in the proof.

\begin{lemma} \label{lemma0} The following inequality holds
\begin{equation}\label{lemma0a}
\tn v \tn^2 \leq C \sum_{K \in \mcK} \tn v \tn_K^2, \quad \text{for all $v \in H^4 + \mcDPN$}
\end{equation}
where $\tn \cdot  \tn_K^2$ is defined by
\begin{equation}\label{lemma0b}
\tn v \tn_K^2 =  h^{-2} | v |_{1,K}^2 + | v |_{2,K}^2 + h^2 | v |_{3,K}^2 + h^4 | v |_{4,K}^2
\end{equation}
\end{lemma}
\begin{proof}
First recall the well known trace inequality
\begin{align} \label{traceineq}
|v|^2_{\partial K} \leq C \left( h^{-1}\|v\|_K^2 + h | v |_{1,K}^2 \right)
\end{align}
which is proven by affinely mapping $K$ to a reference element $\widehat{K}$, using the trace inequalty $\| v \|_{\partial \hat{K}}^2 \leq C \| v \|_{\hat{K}}^2 \| v \|_{1,\hat{K}}^2$ (see \cite{BS}), and finally mapping back to $K$.

Using the triangle inequality on the interior face contributions of \eqref{eq:energynorm} and then using 
the trace inequality Lemma~\ref{lemma0} is readily established.
\qed
\end{proof}

In conformance with \eqref{lemma0b} we also define the energy norm for a set of elements $\mcS$
\begin{align} \label{lemma0extra}
\tn v \tn_{\mcS}^2 = \sum_{K\in\mcS} \tn v \tn_K^2
\end{align}

Furthermore, we will also need to approximate functions using quadratic polynomials on each patch. 
Before we introduce and prove the appropriate estimate for this interpolation error, recall the Bramble-Hilbert lemma given in \cite{BS}.
\begin{lemma} \label{Bramble-Hilbert}
\emph{(Bramble-Hilbert)}
Let $B$ be a ball in $\omega$ such that $\omega$ is star-shaped with respect to $B$ and such that its radius $\rho > (1/2)\rho_{\text{max}}$.
Let $Q^m u$ be the Taylor polynomial of degree $m$ of $u$ averaged over $B$ where $u \in H^m(\omega)$. Then
\begin{align}
|u-Q^m u|_{k,\omega} \leq C_{m,\gamma_\omega} d^{m-k} |u|_{m,\omega} \quad k=0,1,...,m,
\end{align}
where $d = \text{diam}(\omega)$ and $\gamma_\omega$ is the chunkiness parameter of $\omega$.
\end{lemma}
\begin{remark} \label{remark:patchcrit}
The star-shape criterion on $\omega$ means that there should exist a ball $B\in\omega$ such that from any point inside $B$ there is a free line of sight to all points on the boundary of $\omega$. Let $\rho_{\text{max}}$ be the supremum of the radius of all such balls in $\omega$. The chunkiness parameter is then defined by
\begin{align}
\gamma_\omega = \frac{\text{diam}(\omega)}{\rho_{\text{max}}}
\end{align}

We are going to apply the Bramble-Hilbert lemma on each patch, i.e. $\omega = \mcN(K) \cap \mcK$. Further we will need that the chunkiness parameter
for all patches is limited and therefore we introduce the following restriction on the patches: All patches $\mcN(K)\cap\mcK$ fulfill the star-shape criterion and there exists a global constant $\gamma$ such that
\begin{align} \label{eq:patchcrit}
\gamma_{\mcN(K)\cap\mcK} \leq \gamma \quad \text{for all $K\in\mcK$}
\end{align}
Note that the shape regularity of the mesh is not sufficient to guarantee \eqref{eq:patchcrit} for standard four element patches. However, in most cases where the standard patch does not comply to  \eqref{eq:patchcrit} we may add elements to the patch so that it does. Thus, this restriction
will typically not introduce any constraints on the mesh.
\end{remark}

Now we turn to the interpolation error when using quadratic polynomials in the energy norm \eqref{lemma0extra} of a patch $\mcN(K)\cap\mcK$ and present the following lemma.
\begin{lemma}\label{lemma1}
There is a projection operator $P_{2,K}: H^4(\mcN(K) \cap \mcK) \rightarrow \mcP_2 (\mcN(K) \cap \mcK)$
such that
\begin{equation}\label{eq:approx}
\tn u - P_{2,K} u \tn_{\mcN(K) \cap \mcK}  \leq C h \left( \left| u \right|_{3,\mcN(K) \cap \mcK} + h \left| u \right|_{4,\mcN(K) \cap \mcK} \right)
\end{equation}
for all sufficiently smooth $u$.
\end{lemma}

\begin{proof}
By the definition of the energy norm of a patch \eqref{lemma0extra} and the quasi-uniformity of the mesh 
it suffices to show that there exists a patch independent constant $C$ such that
\begin{align}
\left| u - P_{2,K} u \right|_{k,\mcN(K) \cap \mcK}
\leq
C h^{3-k} \left| u \right|_{3,\mcN(K)\cap\mcK}	\quad \text{for $k=1,2,3$}
\end{align}
to prove the lemma. As $P_{2,K} u$ gives zero contribution to fourth order derivatives
the term $C h^2\left| u \right|_{4,\mcN(K) \cap \mcK}$ in \eqref{eq:approx} may be directly
derived from the definition of the energy norm of a patch \eqref{lemma0extra}.

Next we verify that the requirements of Lemma~\ref{Bramble-Hilbert} are fulfilled. By restrictions
on the patches there exists a ball $B$ in every patch such that $\mcN(K)\cap\mcK$ is star-shaped with
respect to $B$.  We let the projection operator $P_{2,K} u$ be defined by the Taylor polynomial of
degree 3 of $u$ averaged over $B$, i.e. $P_{2,K}u = Q^3 u$ as defined in \cite{BS}. This will be a quadratic polynomial.

The constant $C_{m,\omega}$ in Lemma~\ref{Bramble-Hilbert} only depends on the domain through
the chunkiness parameter $\gamma_{\omega}$. As $\omega=\mcN(K)\cap\mcK$ we have from the restriction on the patches \eqref{eq:patchcrit} that there exists a global 
constant $\gamma$ such that $\gamma_{\omega} \leq \gamma$ for all patches.
Using this in the proof of Lemma~\ref{Bramble-Hilbert} in \cite{BS} we have that
\begin{align} \label{limitBSconst}
C_{m,{\gamma_\omega}} \leq C_{m,\gamma}
\end{align}
where $C_{m,\gamma}$ is a patch independent constant and we refer the reader to \cite{BS} for details.

We complete the proof by applying Lemma~\ref{Bramble-Hilbert} together with \eqref{limitBSconst} which gives
\begin{align}
\left| u - P_{2,K} u \right|_{k,\mcN(K) \cap \mcK}
&=
\left| u - Q^3 u \right|_{k,\mcN(K) \cap \mcK} \\
&\leq
 C_{3,\gamma} d^{3-k} |u|_{3,\mcN(K) \cap \mcK} \\
&\leq
%C h_K^{3-k} |u|_{3,\mcN(K) \cap \mcK} \\
%&\leq
C h^{3-k} |u|_{3,\mcN(K) \cap \mcK}
\end{align}
where $C$ is a constant independent of the patch and we used \eqref{finitepatch} and \eqref{eq:quasiu} in the last inequality.
\end{proof}

By Sobolev's inequality pointwise values are well defined for functions in $H^4(\mcK)$
so the Lagrange interpolation operator may be used.
%Thereby the Lagrange interpolation operator is well defined on functions in $H^4(\mcK)$ as $H^4(\mcK) \hookrightarrow H^2(\mcK) \hookrightarrow C^0(\mcK)$.
We extend the standard Lagrange interpolation operator to also define values on ghost elements
outside the domain such that $\pi:C^0(\mcK) \rightarrow \CP_1(\mcK \cup \mcK_G)$. As the
functions we need to interpolate lack support outside the domain the interpolation values at
ghost nodes must be defined. For a ghost node $x_G$ associated with element $K\in\mcK$
we define the interpolation value by
\begin{align} \label{Lagrange1}
(\pi v) (x_G) &= (P_{2,K} v)(x_G) + \Delta_K v
\end{align}
where $\Delta_K v$ is given by
\begin{align} \label{Lagrange2}
\Delta_K v &= (v - P_{2,K} v)|_{x=x_1} + (v - P_{2,K} v)|_{x=x_2} - (v - P_{2,K} v)|_{x=x_3} 
\end{align}
and the numbering of nodes in $K$ is such that $x_3$ is the mirror-symmetric node to $x_G$. Note that $\Delta_K v = 0$ for $v \in \mcP_2(K)$.

We shall also need the following inverse estimate proved in \cite{HansboLarson2002}.
\begin{lemma}\label{lemma:inv}
For all $v \in \mcDPN$ the following estimate hold
\begin{align}
  \label{eq:invtraceineq_mnn}
\sum_{K \in \mcK} h \| \langle M_{nn}(v) \rangle \|^2_{\partial K \setminus (\mcE_S \cup \mcE_F)}  
       \leq C_1 \sum_{K \in \mcK} ( \sigma_{ij}(v), \kappa_{ij}(v) )_K
%\\ \label{eq:invtraceineq_t}
%\sum_{K \in \mcK} h^3 \| \langle T(v) \rangle \|^2_{\partial K  \setminus \mcE_F} \leq C_2 \sum_{K \in \mcK} ( \sigma_{ij}(v), \kappa_{ij}(v) )_K
\end{align}
where $C$ denote a constant independent of the meshsize 
$h$ and the parameter $\beta$.  
\end{lemma}

Finally, we recall the following lemma from \cite{HansboLarson2002} which we will also give proof to.
\begin{lemma}\label{lemma2} Here we collect three basic
results on consistency, continuity, and coercivity:
\\
1. With $u$ the exact solution
of the plate equation and $\mcR U$ the reconstructed dG solution
defined by (\ref{dgrec}) we have
\begin{equation}\label{eq:consistency}
a(u - \mcR U, \mcR v) = 0 \quad {\text{for all $v \in \mathcal{CP}_1$}}.
\end{equation}
2. There is a constant $C$, which is independent of $h$ but in
general depends on $\beta$, such
that
\begin{equation}\label{eq:continuity}
a(v,w) \leq C \tn v \tn \, \tn w \tn \quad v,w \in H^4 + \mcDPN
\end{equation}
\\
3. For $\beta$ sufficiently large the coercivity
estimate
\begin{equation}\label{eq:coercivity}
c \tn v \tn^2  \leq a(v,v) \quad v \in \mcDPN,
\end{equation}
holds, with  a positive constant $c$ independent of
$h$ and $\beta$.
\end{lemma}

\begin{proof}
1.\ This fact is a direct consequence 
of the fact that the exact solution $u$ satisfies the 
variational statement (\ref{galerkin1}). 
\\
2.\ Using the Cauchy Schwarz 
inequality on the definition of the bilinear form 
(\ref{eq:BilinearQuad}) the inequality
\begin{align}
a(v,w) \leq \tn v \tn_{\ast} \tn w \tn_{\ast}
\end{align}
immediately follows where $\tn v \tn^2_{\ast}$ is defined by
\begin{align} \label{extranorm}
\tn v \tn^2_{\ast} = \tn v \tn^2 + \sum_{K \in \mcK} 
  h^{-1} \| [v_{,n}]\|_{\partial K \setminus (\mcE_F \cup \mcE_S)}^2
  +  h^{-3} \| [v]\|_{\partial K \setminus \mcE_F}^2
\end{align}
Estimate \eqref{eq:continuity} follows by showing that the
sum is limited by $\tn v \tn^2$ which we prove next.

We begin by noting that the following equalities hold
\begin{align}
&\| w \|_E^2 = \| w - P_0 w \|_E^2 + \| P_0 w \|_E^2  \\
&\| w - P_0 w \|_E^2 = \frac{h_E^2}{12} \| w_{,t} \|_E^2
\end{align}
for $w \in \mcP_1$. As $[v_{,n}]$ is a linear function we may
apply these equalities to the first term of the sum which together
with quasi-uniformity yields 
\begin{align} \label{eq:exterms2}
h^{-1} \| [v_{,n}]\|_{\partial K \setminus (\mcE_F \cup \mcE_S)}^2 
\leq
C \left( h \| [v_{,nt}] \|_{\partial K \setminus (\mcE_F \cup \mcE_S)}^2
+ h^{-1} \| P_0 [v_{,n}]\|_{\partial K \setminus (\mcE_F \cup \mcE_S)}^2 \right)
\end{align}
where the seconds term already exists in the norm.
We can decompose $v$ into $\tilde{v}+\bar{v}$ where $\tilde{v}\in H^4$ and $\bar{v}\in\mcDPN$.
Due to the jump terms in \eqref{eq:exterms2} the continuous parts of $v$ give no contribution
and we may thus replace $v$ with $\bar{v}$.
To the first term we then apply the triangle inequality to remove the jump
term and can thereby handle each triangle sharing edge $E$ separately.
Applying the trace inequality \eqref{traceineq} we get
\begin{align}
h \| \bar{v}^+_{,nt} \|_{E}^2
\leq
C \left( \| \bar{v}_{,nt} \|_{K^+}^2 + h^2 | \bar{v}_{,nt} |_{1,K^+}^2 \right)
\leq
C  \| v_{,nt} \|_{K^+}^2
\leq
C ( \sigma_{ij}(v), \kappa_{ij}(v) )_{K^+}
\end{align}
where the last inequality comes from that the Lame parameter $\mu > 0$.

For the second term in the sum of \eqref{extranorm} we begin
by subtracting the linear interpolant $\pi [v] = 0$.
%\begin{align}
%h^{-3} \| [v]\|_{\partial K \setminus \mcE_F}^2 = h^{-3} \| [v - \pi v]\|_{\partial K \setminus \mcE_F}^2
%\end{align}
Using the triangle inequality, the trace inequality and interpolation theory we have
\begin{align}
h^{-3} \| v^+ - \pi v^+ \|_{E}^2
&\leq
C \left( h^{-4} \| v - \pi v \|_{K^+}^2 + h^{-2} | v - \pi v |_{1,K^+}^2 \right)
\\ &\leq
C | v |_{2,K^+}^2
\\ &\leq
C ( \sigma_{ij}(v), \kappa_{ij}(v) )_{K^+}
\end{align}
and \eqref{eq:continuity} is established.
\\
3.\
 We have
\begin{multline}
a( v , v )
= \sum_{K \in \mcK} ( \sigma_{ij}(v), \kappa_{ij}(v) )_{K} \\ 
 \quad  - \sum_{E \in \mcE \setminus (\mcE_S \cup \mcE_F)} 
      2 (\langle M_{nn}(v) \rangle , [v_{,n}])_E 
    - \beta h^{-1}  \|P_{l_1} [v_{,n}]\|_{E \setminus (\mcE_S \cup \mcE_F)}^2
\end{multline}
Note that
\begin{align}
(\langle M_{nn}(v) \rangle , [v_{,n}])_E 
      = (\langle M_{nn}(v) \rangle , P_{0} [v_{,n}])_E
\end{align}
since $\langle M_{nn}(v) \rangle$ is a constant and $[v_{,n}]$ is 
a linear function on $E$. Using 
this observation, the Cauchy Schwarz inequality followed by the 
standard inequality $2 ab < \epsilon a^2 + \epsilon^{-1} b^2$, for 
any positive $\epsilon$, and finally the inverse inequality 
(\ref{eq:invtraceineq_mnn}) we obtain
\begin{multline}
 - \sum_{E \in \mcE \setminus (\mcE_S \cup \mcE_F) } 
   2 (\langle M_{nn}(v) \rangle , [v_{,n}])_E  
  \geq \\
	\sum_{K \in \mcK}
	-\epsilon C ( \sigma_{ij}(v), \kappa_{ij}(v) )_{K}
       - \epsilon^{-1} 
     h^{-1} \|P_{0} [v_{,n}]\|^2_{\partial K \setminus (\mcE_S \cup \mcE_F)}
\end{multline}
Given $c$, with $0<c<1$, we choose $\epsilon C = (1-c)/3$ 
and take $\beta \geq c + \epsilon^{-1}$ 
we obtain the coercivity estimate 
(\ref{eq:coercivity}). 
\end{proof}

We are now ready to formulate our main a priori error estimate. 
\begin{theorem} \label{MainTheorem} Assume that the reconstruction operator $\mcR$ is linear and 
satisfies the identity 
\begin{equation}\label{thma}
\mcR_K \pi v = v, \quad \forall v \in \mcP_2(\mcN(K)\cap\mcK), \quad \text{for all $K \in \mcK$}
\end{equation}
where $\pi$ is the extended Lagrange interpolation operator.
Also assume that $u\in H^4(\Omega)$ and that the patch restriction \eqref{eq:patchcrit} is fulfilled.
Then the  following a priori error estimate holds
\begin{equation}
\tn u - \mcR U \tn \leq C h \left( | u |_3 + h | u |_4 \right)
\end{equation}
where $C$ is a constant independent of $h$.
\end{theorem}

Before presenting the proof of Theorem~\ref{MainTheorem} we remark on how the
reconstruction operators presented in Section~\ref{section3} relate to the identity \eqref{thma} in the theorem.
\begin{remark} \label{remark:recon}
By construction the fully quadratic reconstruction and the least squares fully quadratic reconstruction 
satisfy the identity \eqref{thma}. As noted in Section~\ref{sect:StructMesh} this implies that on structured
meshes the Morley reconstruction also satisfies the identity.

On unstructured meshes however, the Morley reconstruction does not satisfy the identity. By the definition
of the Morley basis functions the normal gradient at each edge midpoint must be exactly reconstructed if the 
reconstructed quadratic polynomial shall satisfy \eqref{thma}. As we in the proposed Morley reconstruction use
a pair of linear elements to reconstruct the normal gradient on each edge midpoint, we do not have to consider the complete
patch but rather only pairs of elements.
%Consider the pair of elements illustrated in Figure~X with 
%joint edge $E$ such that $t_E$ is in the $x_2$-direction and the lower edge node defines origo. 
To reconstruct the normal gradient of a quadratic polynomial at the edge midpoint in general five degrees of freedom are needed.
As we in the proposed Morley Reconstruction only use four degrees of freedom, an element pair, to reconstruct the normal gradient
we generally cannot exactly reconstruct quadratic polynomials.

While this remark does not prove that the Morley reconstruction does not converge on unstructured meshes it may give some understanding
of the numerical results.
\end{remark}

\begin{proof}
\emph{of Theorem~\ref{MainTheorem}}
We first note, using the triangle inequality, that
\begin{equation}\label{apriori_a}
\tn u - \mathcal{R} U \tn \leq \tn u - \mathcal{R} \pi u \tn  + \tn \mathcal{R} \pi u  - \mathcal{R} U \tn 
\end{equation}
where $\pi$ is the extended Lagrange interpolation operator.
Using coercivity \eqref{eq:coercivity}, consistency \eqref{eq:consistency}, and the continuity properties in Lemma \ref{lemma2} 
we can estimate the second term as follows
\begin{align}
c \tn \mathcal{R} \pi u - \mathcal{R} U \tn^2 &\leq  a( \mathcal{R} \pi u -  \mathcal{R} U,\mathcal{R} \pi u -  \mathcal{R} U ) 
\\
&=  a( \mathcal{R} \pi u - u  + u -  \mathcal{R} U,\mathcal{R} \pi u -  \mathcal{R} U ) 
\\
&=  a( \mathcal{R} \pi u - u ,\mathcal{R} \pi u -  \mathcal{R} U )
\\ \label{apriori_b}
&\leq C \tn \mathcal{R} \pi u - u \tn \; \tn \mathcal{R} \pi u -  \mathcal{R} U \tn
\end{align}
and thus we arrive at 
\begin{equation}\label{apriori_c}
 \tn \mcR \pi u -  \mathcal{R} U \tn \leq C \tn u - \mathcal{R} \pi u \tn 
\end{equation}
Note that the above derivation follows the proof of Céa's lemma but uses the reconstructions of the analytical and finite element solutions, $\mcR \pi u$ and $\mcR U$, instead of the pure analytical and finite element solutions, $u$ and $U$.
Combining (\ref{apriori_a}) and (\ref{apriori_c}) we obtain
\begin{equation}\label{apriori_d}
 \tn u -  \mathcal{R} U \tn^2 \leq C  \tn u - \mathcal{R} \pi u \tn^2 \leq C \sum_{K \in \mcK} \tn  u - \mathcal{R}_K \pi u \tn_{K}^2
\end{equation}
where we used Lemma \ref{lemma0} in the last inequality.
Adding and subtracting $P_{2,K} u$ and $\mathcal{R}_K \pi P_{2,K} u$ 
and then using the triangle inequality we obtain
\begin{align}
	\tn u-\mathcal{R}_K \pi u \tn_K &\leq 	\tn u - P_{2,K} u \tn_K 
\nonumber \\	
	& \qquad + \tn P_{2,K} u - \mathcal{R}_K \pi P_{2,K} u \tn_{K} 
\nonumber \\	
	& \qquad + \tn \mathcal{R}_K \pi (P_{2,K}  u - u) \tn_{K}
\\
    &= I + II + III \label{terms123}
\end{align}
We now continue with estimates of Terms $I$ to $III$.

\paragraph{Term $\bfI$.} Employing Lemma \ref{lemma1}  we have
 \begin{equation} 
I =\tn u - P_{2,K} u \tn_K  \leq \tn u - P_{2,K} u \tn_{\mcN(K)\cap\mcK} \leq  C ( h| u |_{3,\mcN(K)\cap\mcK} + h^2| u |_{4,\mcN(K)\cap\mcK} )
\end{equation}
 
\paragraph{Term $\bfI\bfI$.} Using the assumption (\ref{thma}) on the reconstruction operator 
we conclude that
\begin{align}
II = \tn P_{2,K} u - \mathcal{R}_K \pi P_{2,K} u \tn_K = 0
\end{align}

\paragraph{Term $\bfI\bfI\bfI$.} Using the following two estimates 
\begin{align}
\label{TermIIIa}
\tn \mcR_K v \tn_{K} &\leq C \tn v \tn_{\mcN(K)} \quad &&\text{for all $v \in \CP_1(\mcN(K))$}
\\
\label{TermIII.b}
\tn \pi (v - P_{2,K} v) \tn_{\mcN(K)} &\leq C \tn  (v - P_{2,K} v) \tn_{\mcN(K)\cap\mcK} \quad &&\text{for all $v \in H^4(\mcN(K)\cap\mcK)$}
\end{align}
which we prove below, we may estimate Term $III$ as follows
\begin{align} 
III &=\tn \mcR_K \pi ({u} - P_{2,K} {u})   \tn_{K} 
\\
&\leq C \tn \pi ({u} - P_{2,K} {u})   \tn_{\mcN(K)}  
\\
&\leq C \tn {u} - P_{2,K} {u} \tn_{\mcN(K)\cap\mcK} 
\\
&\leq C ( h | u |_{3,\mcN(K)\cap\mcK} + h^2 | u |_{4,\mcN(K)\cap\mcK} )
\end{align}
where we used Lemma \ref{lemma1} in the last inequality.

\noindent \emph{Proof of Estimate (\ref{TermIIIa}).} Let $F: \widehat{\mcN(K)} \rightarrow {\mcN(K)}$ 
be a bijective continuous piecewise affine mapping from a reference patch $\widehat{\mcN(K)}$ to the patch 
$\mcN(K)$. We note that, due to shape regularity, we only need to consider a finite number of reference 
patches corresponding to the different topological arrangements of the triangles in the patch. 
The mapping $F$ takes the form 
\begin{equation}
F \widehat{x} = A_{\widehat{K}} \widehat{x} + b_{\widehat{K}},\quad x \in \widehat{K}
\end{equation}

As $F$ maps a triangle of fixed size from a reference patch onto $K$ we have that $|\det{A_{\widehat{K}}}|= C h_K^2$, $\|A_{\widehat{K}}\| \leq h_K/(2\rho_{\widehat{K}}) \leq C h_K$ and by shape regularity $\|A_{\widehat{K}}^{-1}\| \leq h_{\widehat{K}}/(2\rho_{K}) \leq C h_K^{-1}$.
Next we define a mapping  
$\mcF: \mcC \mcP_1(\widehat{\mcN(K)}) \rightarrow \mcC \mcP_1(\mcN(K))$ by
\begin{equation}
 v  = \mcF \widehat{v}  = \widehat{v}  \circ F^{-1} 
\end{equation}
Together with \eqref{eq:quasiu} we have the estimates
\begin{align}\label{IIIa}
| v |_{m,K} &\leq C |\det{A_{{\widehat{K}}}}|^{1/2} \| A^{-1}_{{\widehat{K}}} \|^m | \widehat{v} |_{m,\widehat{K}} \leq C h^{1-m} | \widehat{v} |_{m,\widehat{K}}
\\ \label{IIIb}
| \widehat{v} |_{m,\widehat{K}} &\leq C |\det{A_{{\widehat{K}}}}|^{-1/2} \| A_{{\widehat{K}}} \|^m | v |_{m,K} \leq C h^{m-1} | v |_{m,K}
\end{align}

%\begin{align}\label{IIIa}
%| v |_{m,K} &\leq C ( \text{det}\, A_{\widehat{K}} )^{1/2} \| A^{-1}_{\widehat{K}} \|^m | \widehat{v} |_{m,\widehat{K}} \leq C h_K^{1-m} | \widehat{v} |_{m,\widehat{K}} 
%\\ \label{IIIb}
%| \widehat{v} |_{m,\widehat{K}} &\leq C ( \text{det}\, A_{\widehat{K}}  )^{-1/2} \| A_{\widehat{K}} \|^m | v |_{m,K} \leq C h_K^{m-1} | v |_{m,K}
%\end{align}
Using (\ref{IIIa}) and (\ref{IIIb}) we conclude that there are constants $c$ and $C$ such that
\begin{align}\label{IIIc}
\tn v \tn^2_K &\leq c \tn \widehat{v} \tn^2_{\widehat{K}}  \\
\tn \widehat{v} \tn^2_{\widehat{\mcN(K)}} &\leq C \tn v \tn^2_{\mcN(K)} \label{IIIcc}
\end{align}
where
\begin{equation}
\tn \widehat{v} \tn_{\widehat{K}}^2 =  h^{-2} \left( | \widehat{v} |_{1,\widehat{K}}^2 +  | \widehat{v} |_{2,\widehat{K}}^2 +  | \widehat{v} |_{3,\widehat{K}}^2 \right)
 % \tn \widehat{v} \tn^2_{\widehat{K}}  =  \sum_{m=1}^3 |\widehat{v} |^2_{m,\widehat{K}}  
\end{equation}

Returning to the proof of (\ref{TermIIIa}) we first show that the inequality holds on the reference neighborhood
\begin{equation}\label{IIId}
\tn \widehat{\mcR_K v} \tn_{\widehat{K}} \leq C \tn \widehat{v} \tn_{\widehat{\mcN(K)}} \quad \forall v \in \mcC \mcP_1(\widehat{\mcN(K)})
\end{equation} 
We note that $\tn \widehat{v} \tn_{\widehat{\mcN(K)}} = 0$ if and only if $\widehat{v}$ is 
constant on $\widehat{\mcN(K)}$ but then $v = \mcF \widehat{v}$ is also constant on $\mcN(K)$ and thus 
$\mcR_K v = v$ is also constant. Therefore we conclude that $\tn \widehat{\mcR_K v} \tn_{\widehat{K}} = 0$ 
if $\tn \widehat{v} \tn_{\widehat{\mcN(K)}} = 0$ and inequality (\ref{IIId}) thus follows from finite 
dimensionality. Combining (\ref{IIIc}), \eqref{IIIcc} and (\ref{IIId}) we get 
\begin{equation} \label{est3a}
 \tn \mcR_K v \tn_{K}  \leq C \tn \widehat{\mcR_K v} \tn_{\widehat{K}} \leq C \tn \widehat{v} \tn_{\widehat{\mcN(K)}}\leq C \tn v \tn_{\mcN(K)}
\end{equation}
which concludes the proof of estimate (\ref{TermIIIa}).

\noindent \emph{Proof of Estimate (\ref{TermIII.b}).}
Let $w = v - P_{2,K}v$. By contruction of the extended Lagrange interpolant (\ref{Lagrange1}, \ref{Lagrange2}) and mirror symmetry of ghost elements we have
\begin{align}
\tn \pi w \tn^2_{\mcN(K)} \leq C \tn \pi w \tn^2_{\mcN(K)\cap\mcK}
\end{align}
Adding and subtracting $w$, using the triangle inequality 
and interpolation error estimates we get
\begin{align}
\tn \pi w \tn^2_{\mcN(K)\cap\mcK} &= \sum_{ {K} \subset \mcN(K)\cap\mcK } h^{-2} | \pi w |^2_{1,{K}}
\\
& \leq C \sum_{ {K} \subset \mcN(K)\cap\mcK }  h^{-2} | w - \pi w |^2_{1,{K}}  + h^{-2}| w |^2_{1,{K}}
\\
& \leq   C \sum_{ {K} \subset \mcN(K)\cap\mcK }  h^{-2} | w |^2_{1,{K}} + | w |^2_{2,{K}}
\\
&\leq C \tn w \tn_{\mcN(K)\cap\mcK}^2
\end{align}
and thus estimate (\ref{TermIII.b}) follows.

We have thereby completed the estimates of Terms $I$ to $III$ in \eqref{terms123}. Using these in \eqref{apriori_d} we thus have
\begin{align}
\tn u - \mcR U \tn^2 &\leq C \sum_{K\in\mcK} (I+II+III)^2
\\ &\leq
C \sum_{K\in\mcK} \left(h |u|_{3,\mcN(K)\cap\mcK} + h^2 |u|_{4,\mcN(K)\cap\mcK} \right)^2
\\ &\leq
C \sum_{K\in\mcK} h^2 |u|_{3,\mcN(K)\cap\mcK}^2 +  h^4 |u|_{4,\mcN(K)\cap\mcK}^2
\end{align}
By shape regularity we have that the number of overlaps in the sum will be finite and thereby there exists a constant $C$ such that
\begin{align}
\tn u - \mcR U \tn^2
\leq
C \left( h^2 |u|_{3}^2 +  h^4 |u|_{4}^2 \right)
\leq
C \left( h |u|_{3} +  h^2 |u|_{4} \right)^2
\end{align}
%and hence
%\begin{align}
%\tn u - \mcR U \tn \leq  C ( h |u|_{3} + h^2 |u|_{4} )
%\end{align}
which completes the proof.
\end{proof}

% L2 estimate

We now turn to an estimate of the $L^2$ norm of the error.
This is derived using a duality argument (Nitsche's trick).
We assume that for all $\psi \in H^4 + \mcDPN$ there is a $\phi \in H^4$ such that
\begin{equation}\label{eq:dual}
a(v,\phi) = (v, \psi), \quad \text{for all $v \in H^4 + \mcDPN$}
\end{equation}
and that the following stability estimate holds
\begin{equation}\label{eq:dual_stab}
\| \phi \|_4 \leq C \| \psi \|
\end{equation}
On smooth domains and convex bounded polygonal domains where the inner angle at each corner is less than $126.3^\circ$ this assumption is true, see \cite{Blum}.
\begin{theorem}\label{theorem:2} If the stability estimate 
(\ref{eq:dual_stab}) holds, then $U$ satisfies
\begin{align}\label{eq:err_est}
\| u - \mcR U \| \leq C 
    h^2 \left( | u |_{3} + h | u |_{4} \right)
\end{align}
for sufficiently regular $u$. The constant $C$ is 
independent of $h$ but may in general depend on $\beta$.
\end{theorem}
\begin{proof}
Setting $v = \psi = u - \mcR U$, in the dual problem (\ref{eq:dual}) 
and using consistency \eqref{eq:consistency} to subtract the
reconstruction $\mcR \pi \phi$ of $\pi \phi$ we obtain 
\begin{align}
\| u - \mcR U \|^2 &= a(u - \mcR U,\phi) 
\\
&=a(u - \mcR U ,\phi - \mcR \pi \phi)
\\
& \leq C \tn u - \mcR U \tn \, \tn \phi - \mcR \pi \phi \tn
\end{align}
where we used continuity (\ref{eq:continuity}) in the last step. 
Next using Theorem \ref{MainTheorem} and results \eqref{apriori_d} in its proof we have 
\begin{align}
\| u - \mcR U \|^2 \leq  C h^2
    (\left| u \right|_{3} + h \left| u \right|_{4}) ( \left| \phi \right|_{3} + h \left| \phi \right|_{4} )
\end{align}
which together with the stability estimate (\ref{eq:dual_stab}) 
concludes the proof.
\end{proof}

%%%%%%%%%%%%%%

In Theorem \ref{MainTheorem} and Theorem \ref{theorem:2} we have given a priori error estimates for the
reconstructed solution $\mcR U$ in energy norm and in $L^2$ norm. We now turn to showing an a priori error estimate for
the continuous piecewise linear solution $U$ in $L^2$ norm.
\begin{theorem}\label{theorem:3} If the stability estimate 
(\ref{eq:dual_stab}) holds, then $U$ satisfies
\begin{align}
\| u - U \| \leq C
    h^2 \left(\left| u \right|_2 + \left| u \right|_3 + h \left| u \right|_4 \right)
\end{align}for sufficiently regular $u$. The constant $C$ is 
independent of $h$ but may in general depend on $\beta$.
\end{theorem}
\begin{proof}
Using triangle inequality we have
\begin{align}
\| u - U \| \leq \| u - \mcR U \| + \| \mcR U - U \|
\end{align}
where the first term is evaluated by Theorem \ref{theorem:2}. For the second term we use a standard interpolation estimate
\begin{align}
\| \mcR U - U \| = \| \mcR U - \pi \mcR U \| \leq C h^2 |\mcR U|_2 \leq C h^2 \left( |u - \mcR U|_2 + |u|_2 \right)
\end{align}
where we in the last inequality use the triangle inequality on the seminorm.

As the Lamé parameter $\mu > 0$ there exists a constant $C$ such that
\begin{align}
|u - \mcR U|_2^2 = \sum_{K \in \mcK} |u - \mcR U|_{2,K}^2 \leq C \tn u - \mcR U \tn^2
\end{align}
which is limited by Theorem \ref{MainTheorem}. This gives the error estimate
\begin{align}
\| u - U \| \leq C h^2 \left( |u|_2 + |u|_3 + h |u|_4 \right)
\end{align}
which concludes the proof.
\end{proof}

%%%%%%%%%%%%%%%%%%%%%%%%%%%%%%%%%%%%%%%%%%%%%%%%%%%%%%%%%%%%%%%%%%%%%%%%%%%%%%%%%%%%%%%%%%%%%%%%%%%%%%%%%%%%%%

\section{Numerical results}
Numerical results will be presented for the following proposed methods: Morley reconstruction, fully quadratic reconstruction, and least squared fully quadratic reconstruction. Also, for comparison we will present results for: the Basic Plate Triangle, the nonconforming Morley triangle, a quadratic continuous/discontinuous Galerkin method featuring $C^0$ continuity, and a quadratic discontinuous Galerkin method continuous at the mesh nodes.

Note that for the reconstruction methods, the pointwise error is defined as $e=u-\mcR U$ unless otherwise stated. For other methods the pointwise error is as usual defined as $e=u-U$.

\subsection{Model Problems}
To study the convergence properties of the proposed methods we use two model problems where analytical solutions are known. 

\subsubsection{Problem 1: Simply Supported Plate under Sinusoidal Load}
Consider a simply supported unit square plate, $\Omega=[0,1]^2$, with $D=1$ and $\nu=0$. Find the deflection $u$ given the sinusoidal load
\begin{align}
f =	25 \pi^4 \sin(\pi x) sin(2 \pi y)
\end{align}
This problem has the analytical solution $u = \sin(\pi x) \sin(2 \pi y)$.

\subsubsection{Problem 2: Mixed Boundary Conditions with Uniform Load}
Consider a unit square plate, $\Omega=[0,1]^2$, with two opposite sides simply supported, one side clamped, and the last side free. Given $E=10^6$, $t=0.01$, $\nu = 0.3$ and a uniform load $f=1$, find the deflection $u$ of the plate. An analytical solution in the form of a series expansion is given in Example 46 in~\cite{Timoshenko}.

\subsection{Mesh}
The triangulations we consider include both structured and unstructured meshes. The structured meshes conform to the criteria discussed in Section~\ref{sect:StructMesh}. Example triangulations of the unit square for both structured and unstructured meshes are illustrated in Figure~\ref{fig:meshes}.

\begin{figure}
\centering
 \subfigure[Structured mesh.]{
   \includegraphics[width=0.45\linewidth]{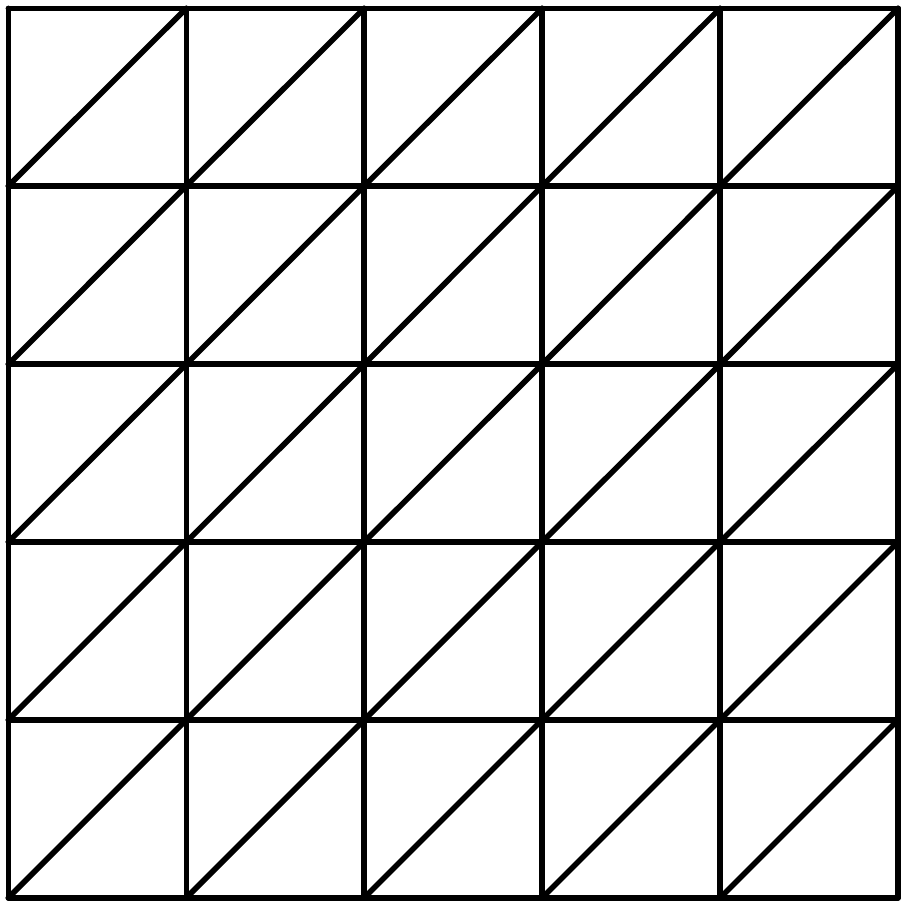}
	\label{fig:structuredmesh}
 }
 \subfigure[Unstructured mesh.]{
   \includegraphics[width=0.45\linewidth]{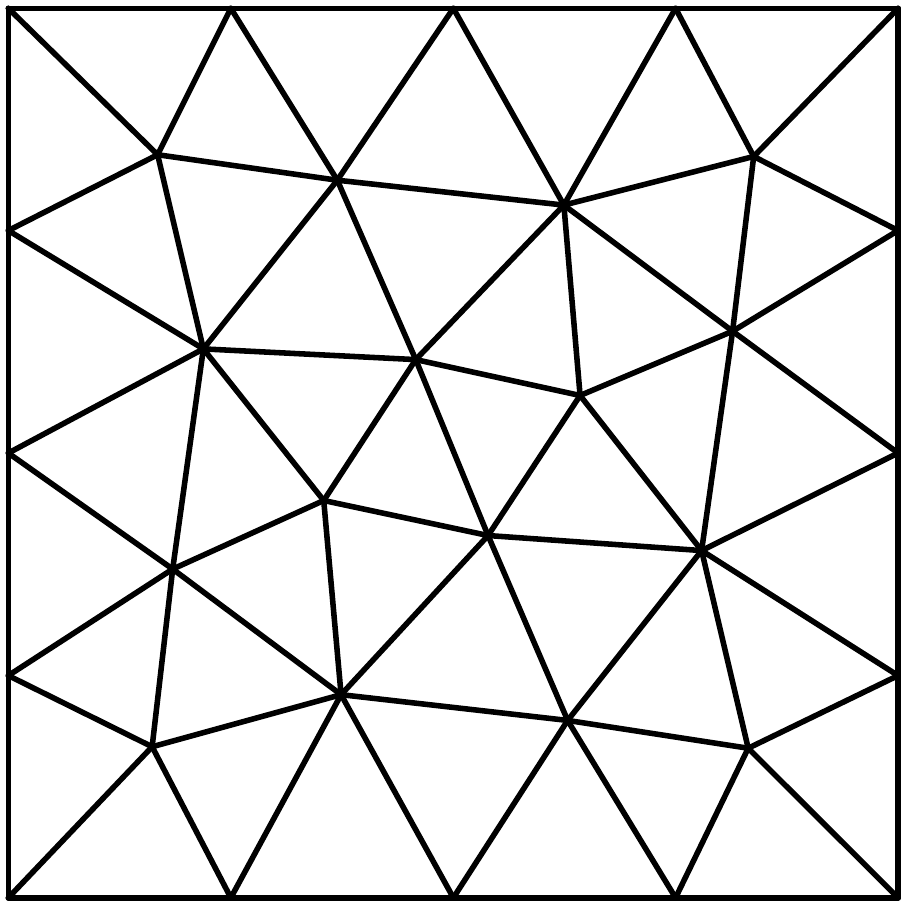}
	\label{fig:unstructuredmesh}
 }
\caption{Two example triangulations of the unit square with comparable mesh size $h$. The left triangulation (a) is a structured mesh and the right triangulation (b) is an unstructured mesh.}
\label{fig:meshes}
\end{figure}

\subsection{Numerical Examples}
To illustrate interesting features of the proposed methods we here give a few numerical solutions.

\subsubsection{Nodal Continuity and Continuity of Normal Gradient}
A reconstructed solution to Problem 1 on a coarse mesh is presented in Figure~\ref{fig:reconstruction}. Note that continuity of the nodes is strongly enforced and the continuity of the normal gradients on edge midpoints is weakly enforced through the dG method's \eqref{eq:BilinearQuad} inherent penalization of jumps in the normal gradient.

\begin{figure}
\centering
\includegraphics[width=12cm]{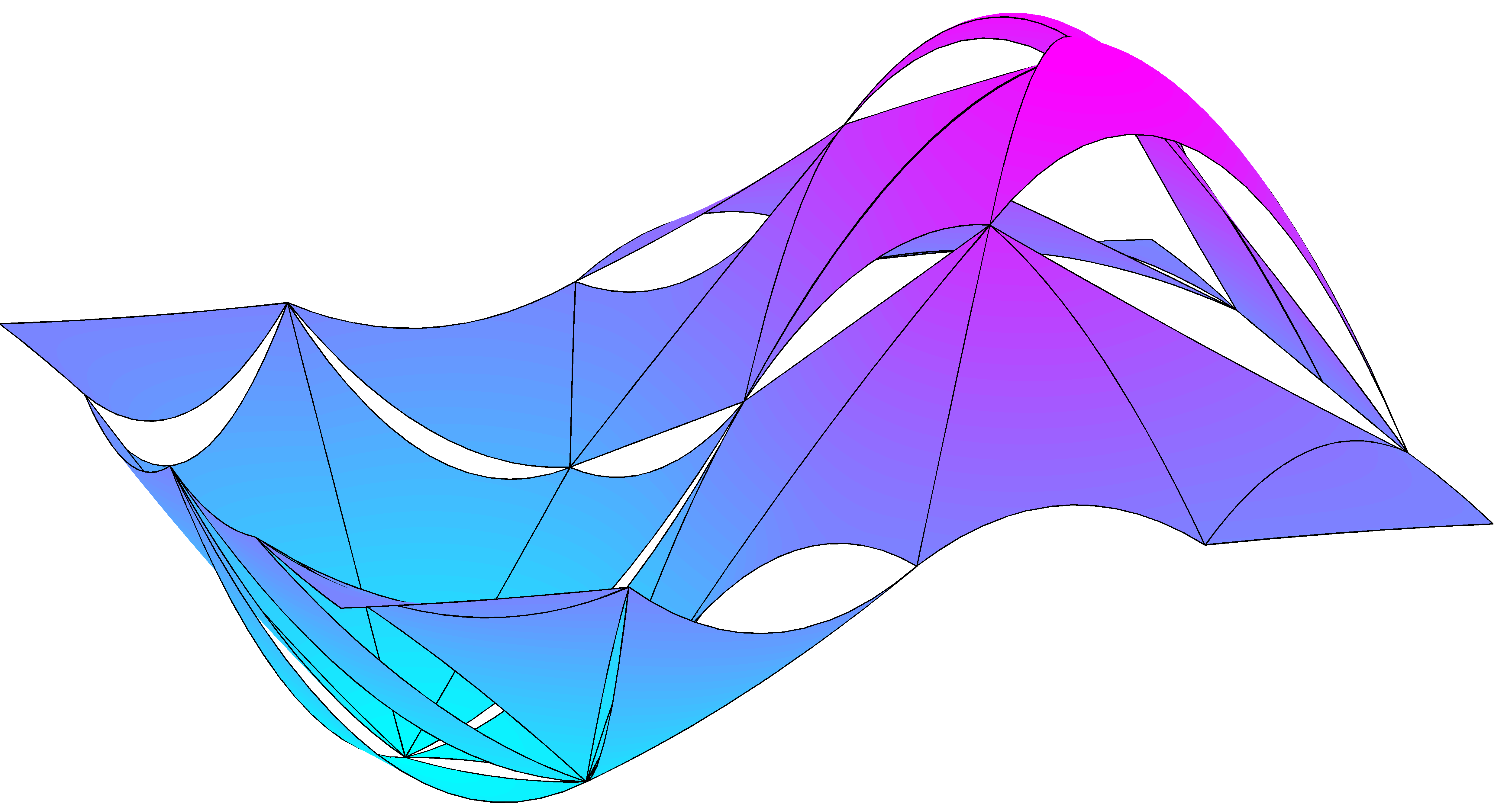}
\caption{Example dG solution to Problem 1 using fully quadratic reconstruction on a coarse unstructured mesh.}
\label{fig:reconstruction}
\end{figure}

\subsubsection{Solution on Mesh including Degenerate Patch}
To illustrate the need of the least squared fully quadratic reconstruction we use a mesh which include a degenerate patch, see Figure~\ref{fig:degenmesh}. This mesh was modified to include this patch as it is unlikely that degenerate patches appear when using quality mesh generation. The collapsed solution is shown in Figure~\ref{fig:degencollapsed}. By extending the patch as in Figure~\ref{fig:degenmeshext} the least squares fully quadratic reconstruction gives an accurate solution, see Figure~\ref{fig:degenfixed}.

\begin{figure}
\centering
 \subfigure[Degenerate four triangle patch.]{
\includegraphics[width=0.45\linewidth]{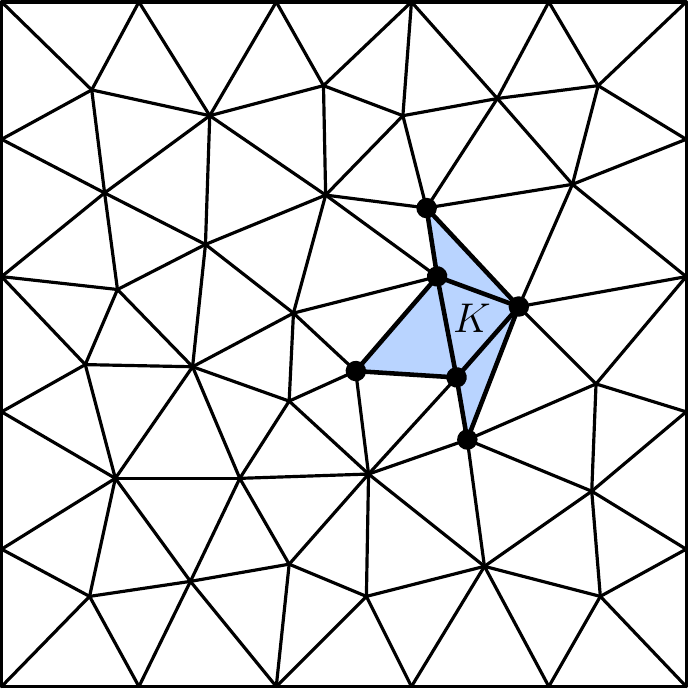}
   %\import{illustrations/}{mesh_degen.eps_tex}
	\label{fig:degenmesh}
 }
 \subfigure[Extended patch.]{
\includegraphics[width=0.45\linewidth]{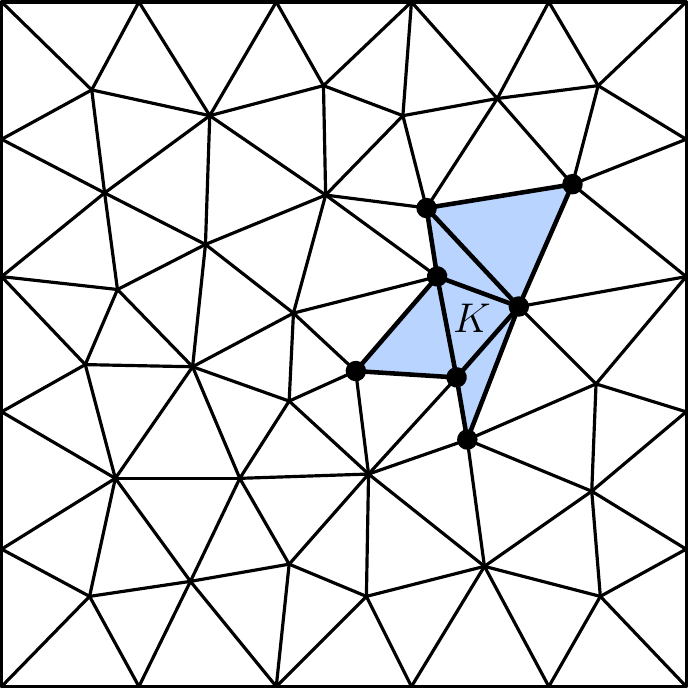}
   %\import{illustrations/}{mesh_ls.eps_tex}
	\label{fig:degenmeshext}
 }
\caption{Example mesh that includes a degenerate patch indicated in (a) and an extension of that patch indicated in (b).}
\label{fig:degeneratemesh}
\end{figure}

\begin{figure}
\centering
 \subfigure[Collapsed solution due to degenerate patch.]{
	\includegraphics[width=10cm]{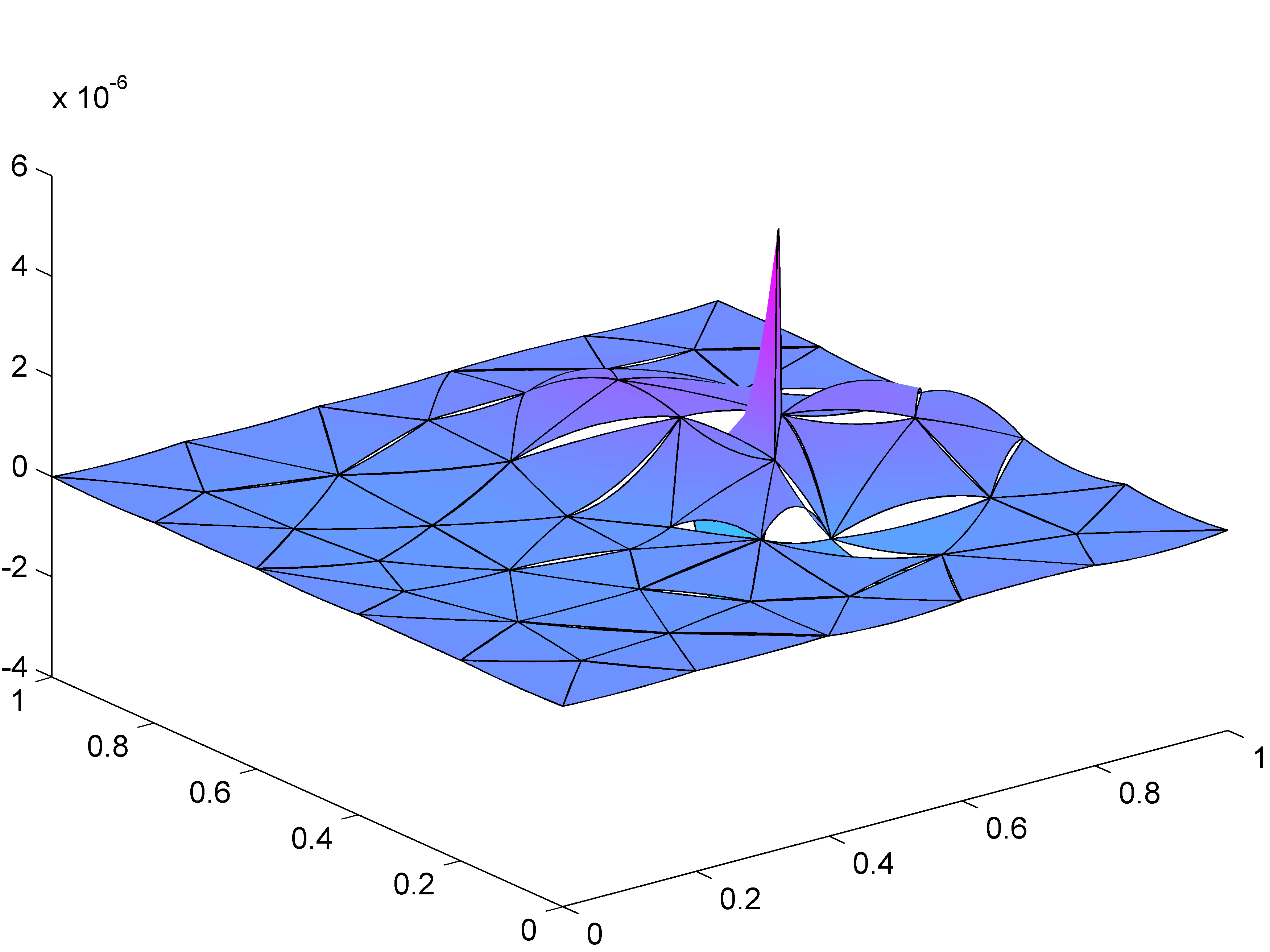}
	\label{fig:degencollapsed}
 }
 \subfigure[Accurate solution using extended patch.]{
	\includegraphics[width=10cm]{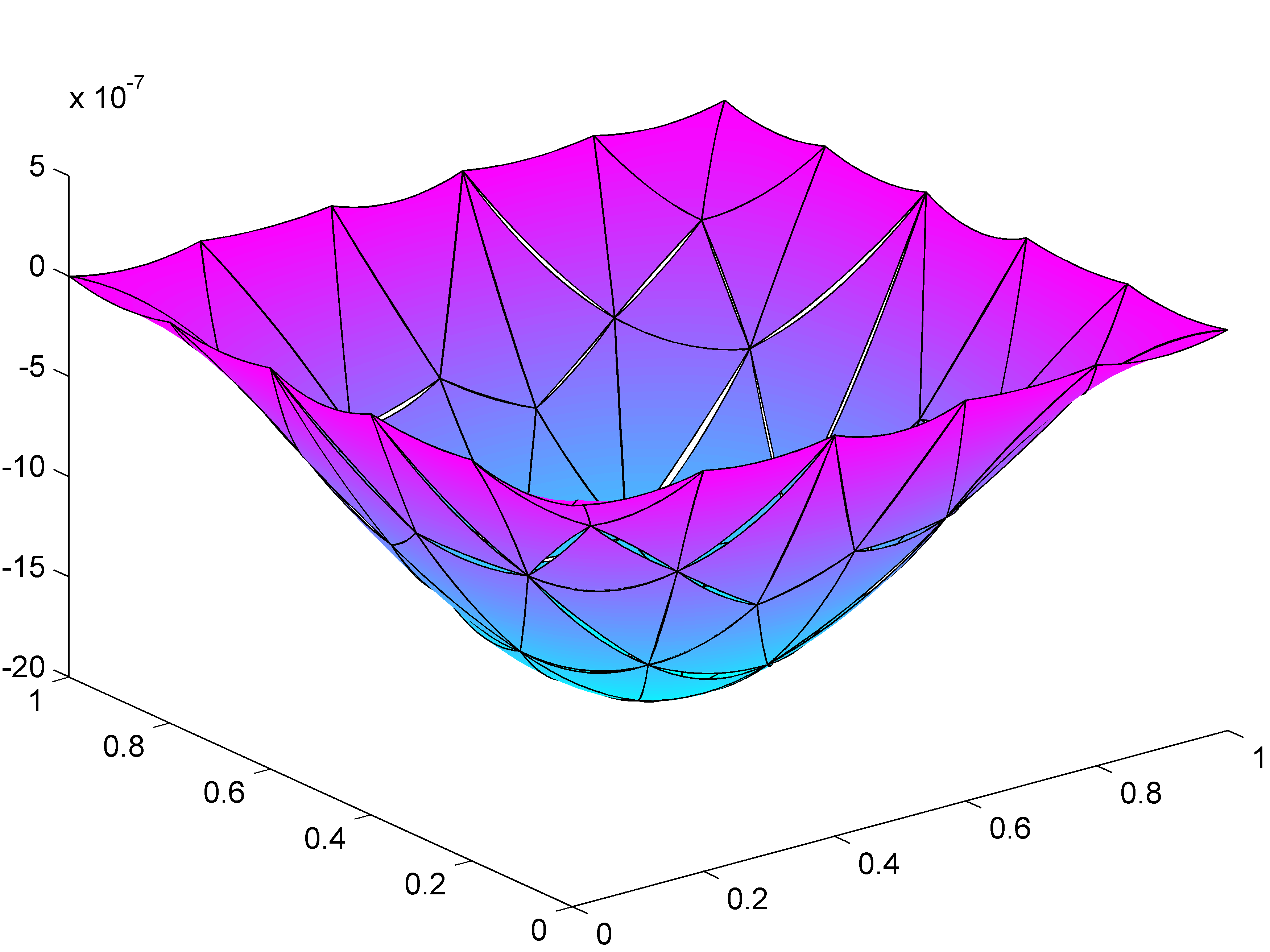}
	\label{fig:degenfixed}
 }
\caption{Numerical solution for a simply supported plate under uniform load using LSFQ-reconstruction including the degenerate patch is shown in (a) and including the extended patch is shown in (b).}
\label{fig:degeneratesolution}
\end{figure}

\subsection{Convergence}

We consider convergence in both the energy norm \eqref{eq:energynorm} and in the $L^2$ norm. As the nonconforming Morley plate can be viewed as a special case of the quadratic discontinuous Galerkin method continuous at the nodes where $\beta\rightarrow\infty$ in \eqref{eq:BilinearQuad}, the energy norm is applicable also to this element.

\subsubsection{Comparison of Morley Reconstruction and Basic Plate Triangle}
As shown in Section~\ref{sect:BPTequiv} the major difference between the Morley reconstruction and the Basic Plate Triangle \cite{OnateCervera1993} lies in the calculation of the load vector. A comparison of the two methods using Problem 1 on a structured mesh is shown in Figure~\ref{fig:BPTvsMorley} and clearly indicate a better convergence rate when using the load calculation of the reconstructed Morley method. The difference in enforcement of clamped boundary conditions does not produce any noticible difference in numerical results. While keeping the difference in convergence rate in mind, we will from here on let the results for the Morley reconstruction method also represent the beviour of the Basic Plate Triangle.

\begin{figure}
\centering
\includegraphics[width=10cm]{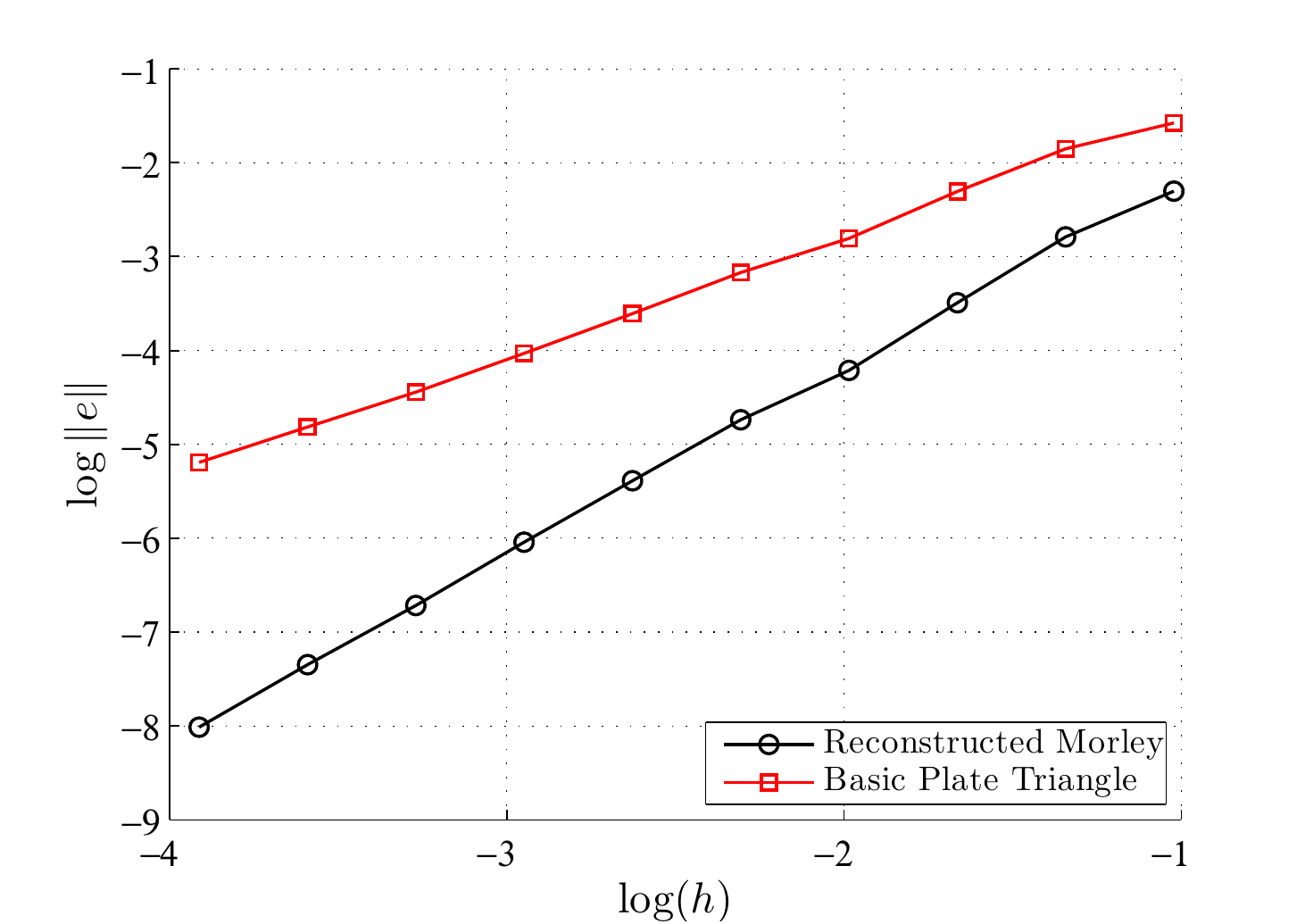}
\caption{The error in the numerical solution of Problem 1 versus the mesh size $h$ on structured meshes. Slopes for the Basic Plate Triangle and the Reconstructed Morley are 1.28 and 1.99 respectively. The error $e=u-U$ is measured in $L^2$ norm.}
\label{fig:BPTvsMorley}
\end{figure}

\subsubsection{Convergence on Structured and Unstructured Meshes}

As noted in Section~\ref{sect:StructMesh} the Morley reconstruction and the fully quadratic reconstruction coincide on structured meshes, and should thereby produce identical results. This is seen in the convergence plots for structured meshes, Figures~\ref{fig:P1_structured_E} and \ref{fig:P1_structured_L2}, where their paths overlap.

On unstructured meshes the Morley reconstruction/Basic Plate Triangle does not converge to the analytical solution. This is seen in Figures~\ref{fig:P1_regular_E}-\ref{fig:p2unstructuredL2}. As noted in Remark~\ref{remark:recon} the Morley reconstruction does not fulfill the assumption of Theorem~\ref{MainTheorem} on unstructured meshes, and thus the a priori error estimates are not valid. On the other hand, the fully quadratic reconstruction does show optimal convergence on unstructured meshes, as predicted by the a priori estimates. In the figures slopes close to $1$ for the error in energy norm and slopes close to $2$ for the error in $L^2$ norm indicate optimal convergence. With the noted exception of the Morley reconstruction/Basic Plate Triangle on unstructured meshes, Figures~\ref{fig:P1_structured_E}-\ref{fig:p2unstructuredL2} indicate optimal convergence for all the compared methods.

We have previously mentioned that the nonconforming Morley triangle can be seen as a special case of the quadratic nodal continuous discontinuous Galerkin method. This is natural as the $\beta$ penalty parameter in the dG method enforces continuity of the normal derivatives over each edge midpoint, which is the very definition of the Morley basis functions. As shown in Figures~\ref{fig:P1_structured_E}-\ref{fig:p2unstructuredL2}, the convergence results for the respective method are close to identical for $\beta=100$ as used in these calculations.

\begin{figure}
\centering
\includegraphics[width=10cm]{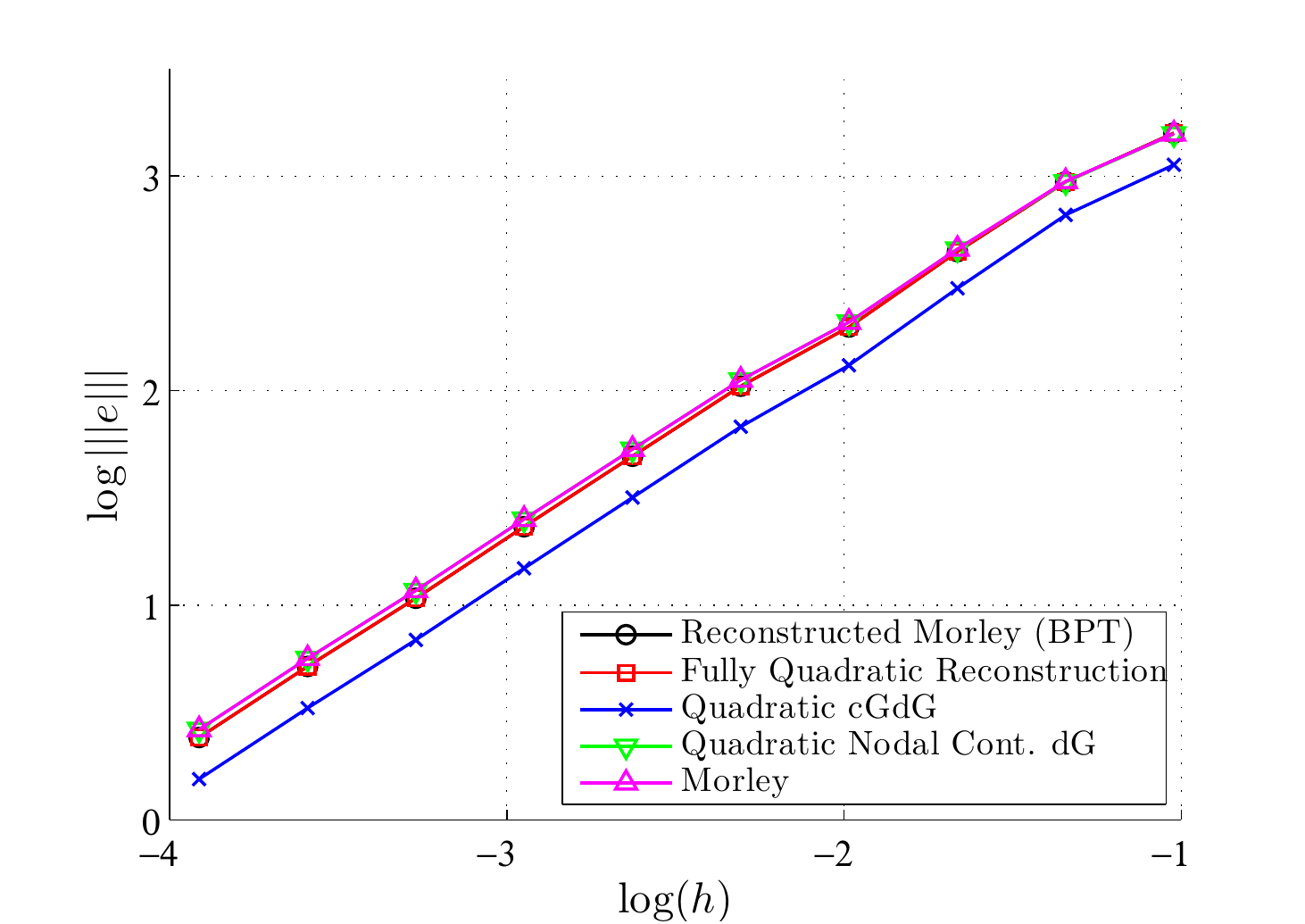}
\caption{The error in the numerical solution of Problem 1 versus the mesh size $h$. Structured meshes are used and the error $e$ is measured in energy norm. Note that the reconstructed Morley and fully quadratic reconstruction produce identical results.}
\label{fig:P1_structured_E}
\end{figure}

\begin{figure}
\centering
\includegraphics[width=10cm]{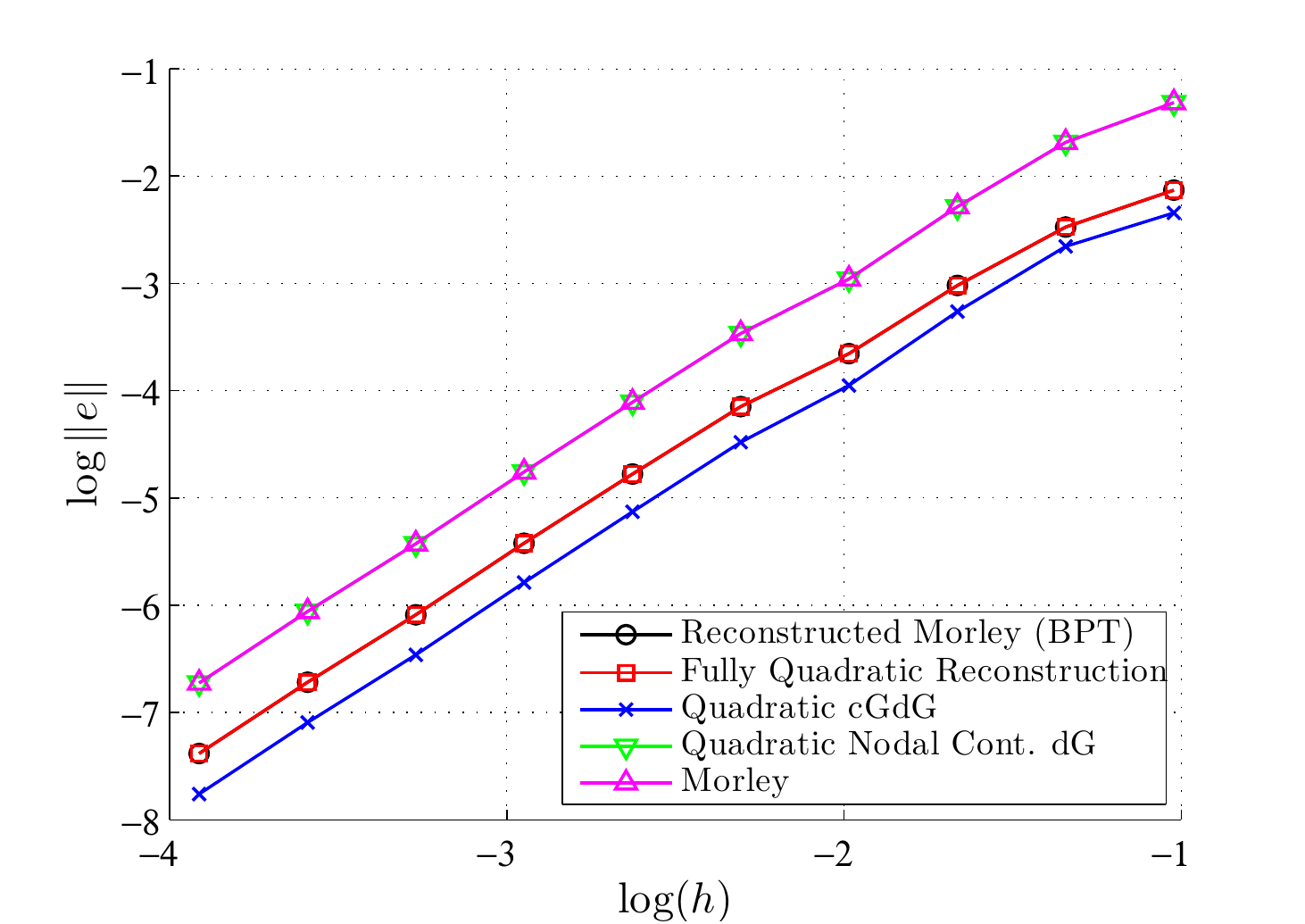}
\caption{The error in the numerical solution of Problem 1 versus the mesh size $h$. Structured meshes are used and the error $e$ is measured in $L^2$ norm. Note that the reconstructed Morley and fully quadratic reconstruction produce identical results.}
\label{fig:P1_structured_L2}
\end{figure}

\begin{figure}
\centering
\includegraphics[width=10cm]{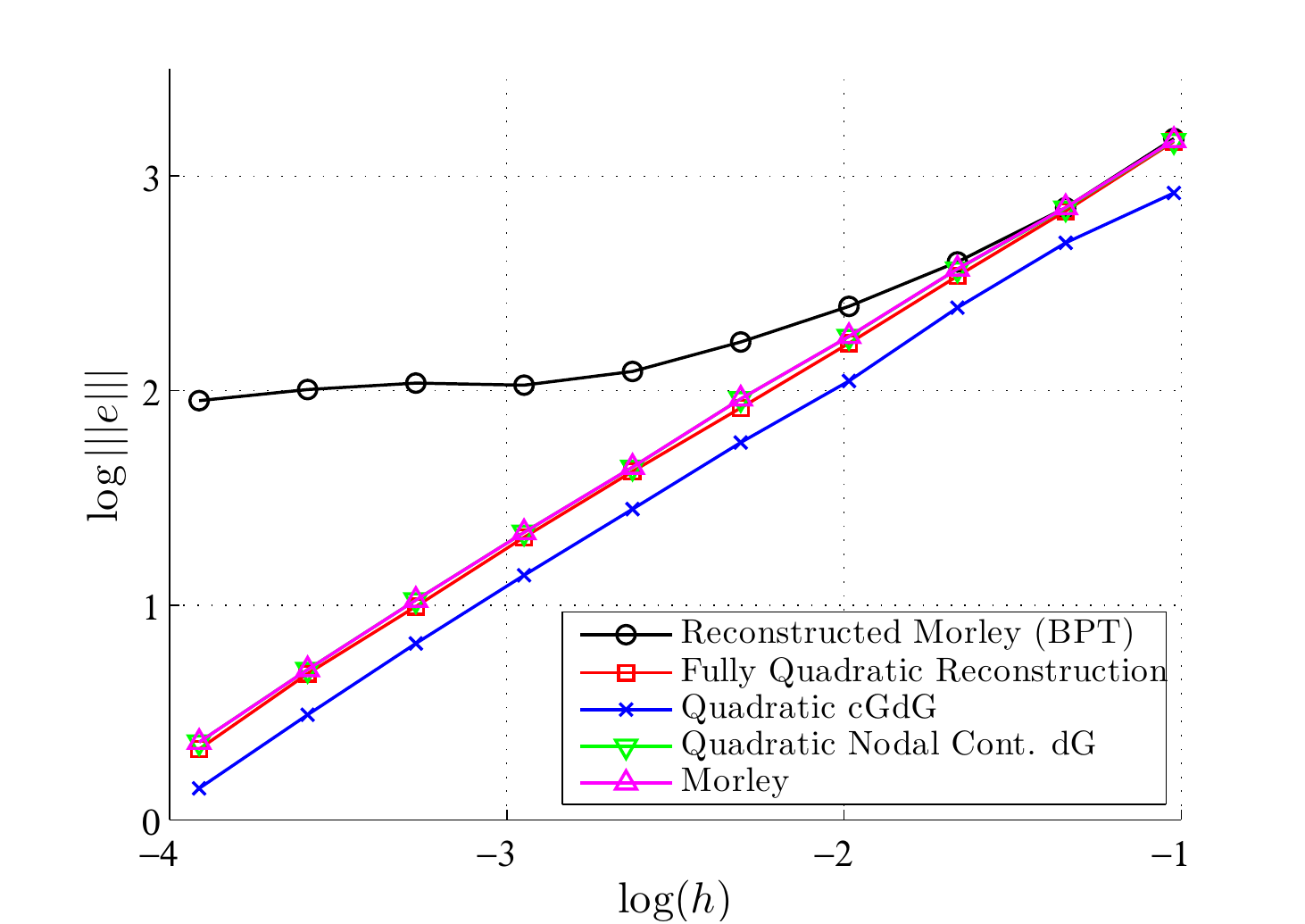}
\caption{The error in the numerical solution of Problem 1 versus the mesh size $h$. Unstructured meshes are used and the error $e$ is measured in energy norm.}
\label{fig:P1_regular_E}
\end{figure}

\begin{figure}
\centering
\includegraphics[width=10cm]{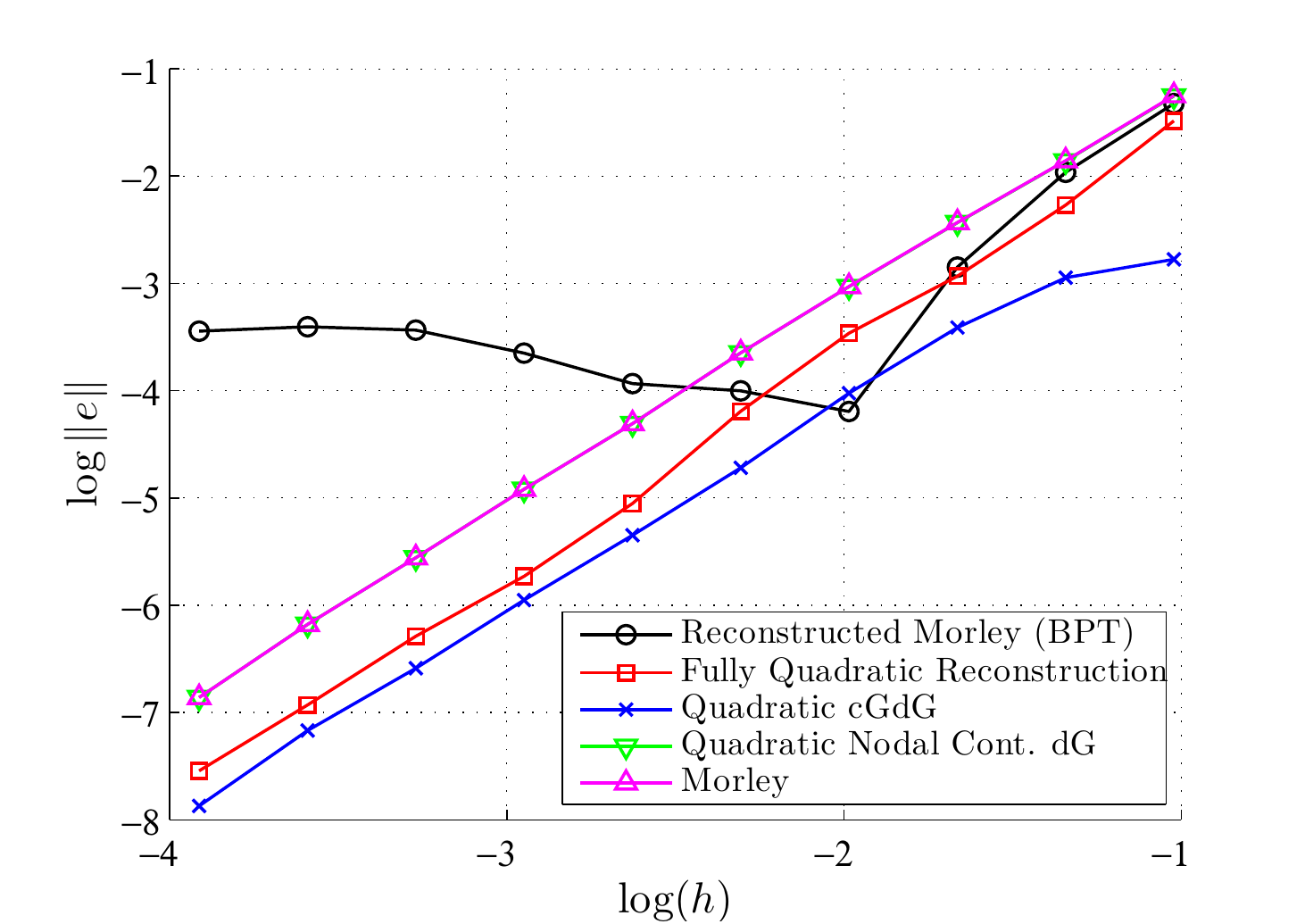}
\caption{The error in the numerical solution of Problem 1 versus the mesh size $h$. Unstructured meshes are used and the error $e$ is measured in $L^2$ norm.}
\label{fig:P1_regular_L2}
\end{figure}

\begin{figure}
\centering
\includegraphics[width=10cm]{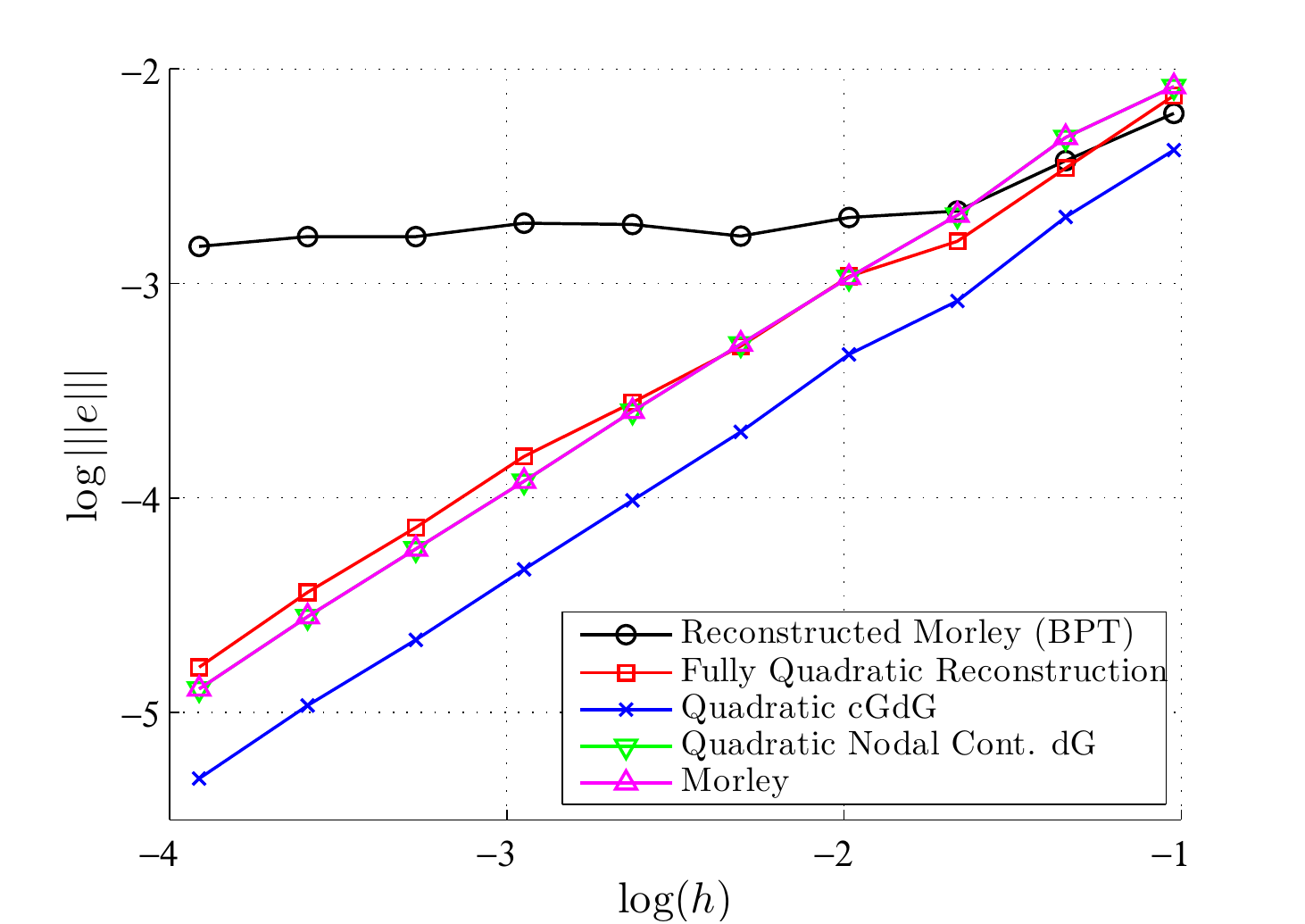}
\caption{The error in the numerical solution of Problem 2 versus the mesh size $h$. Unstructured meshes are used and the error $e$ is measured in energy norm.}
\label{fig:p2unstructuredE}
\end{figure}

\begin{figure}
\centering
\includegraphics[width=10cm]{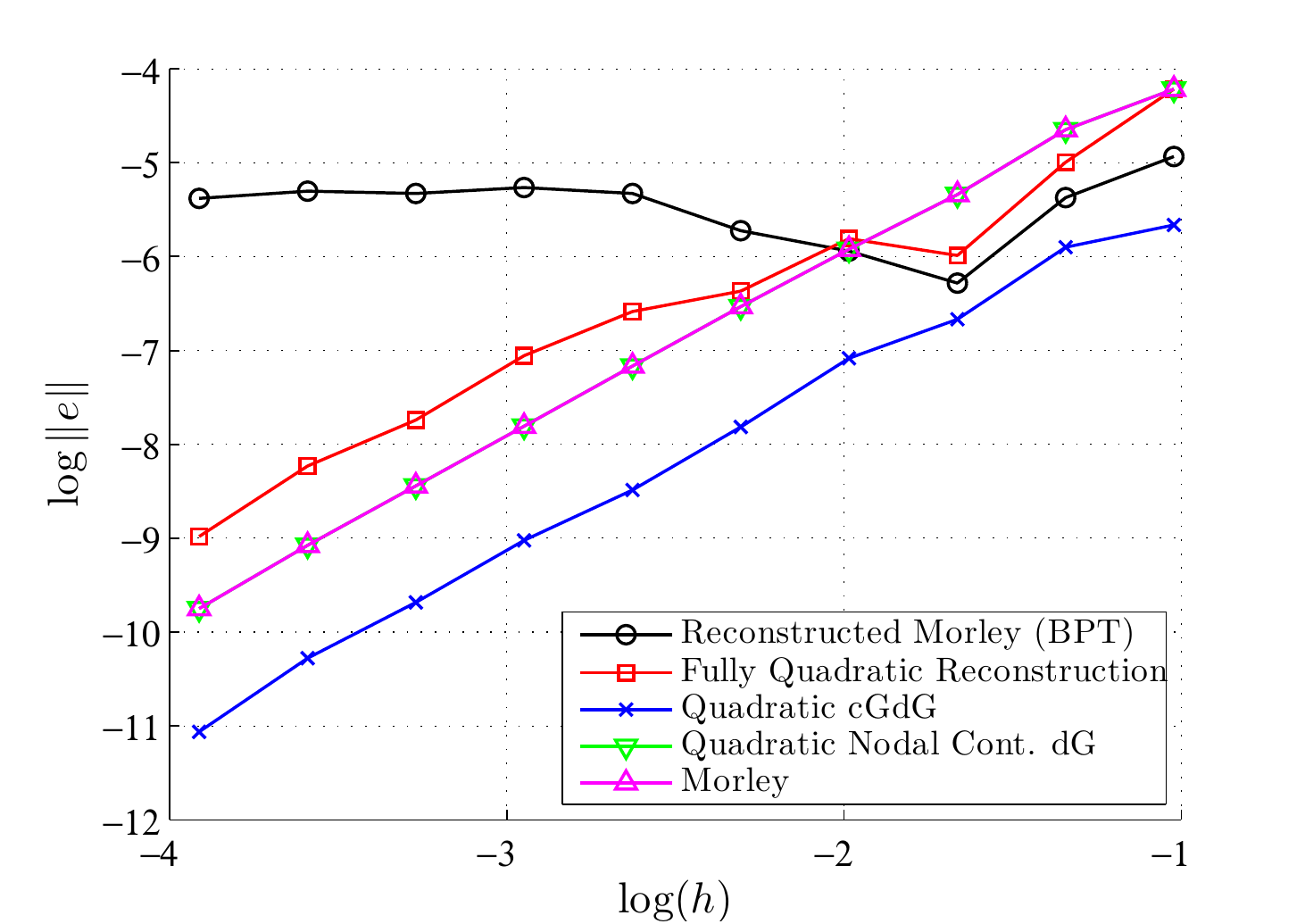}
\caption{The error in the numerical solution of Problem 2 versus the mesh size $h$. Unstructured meshes are used and the error $e$ is measured in $L^2$ norm.}
\label{fig:p2unstructuredL2}
\end{figure}

\subsubsection{Number of Degrees of Freedom}

To give some indication of the performance of these elements in regards to how many degrees of freedom are needed to represent the solution we give Figure~\ref{fig:p2unstructuredE_ndofs}. While it is seen in Figures \ref{fig:P1_structured_E}-\ref{fig:p2unstructuredL2} that the quadratic cG/dG method has the best performance among the tested methods with respect to mesh discretization, Figure~\ref{fig:p2unstructuredE_ndofs} indicates that the fully quadratic reconstruction has the most compact representation performance wise. This is natural as we have smooth solutions. Even though the quadratic nodal continuous dG method produce results close to identical to those of the Morley triangle with regards to mesh discretization it does feature two degrees of freedom on each edge midpoint compared to one for the Morley triangle, explaining that more degrees of freedom are needed for par performance.

\begin{figure}
\centering
\includegraphics[width=10cm]{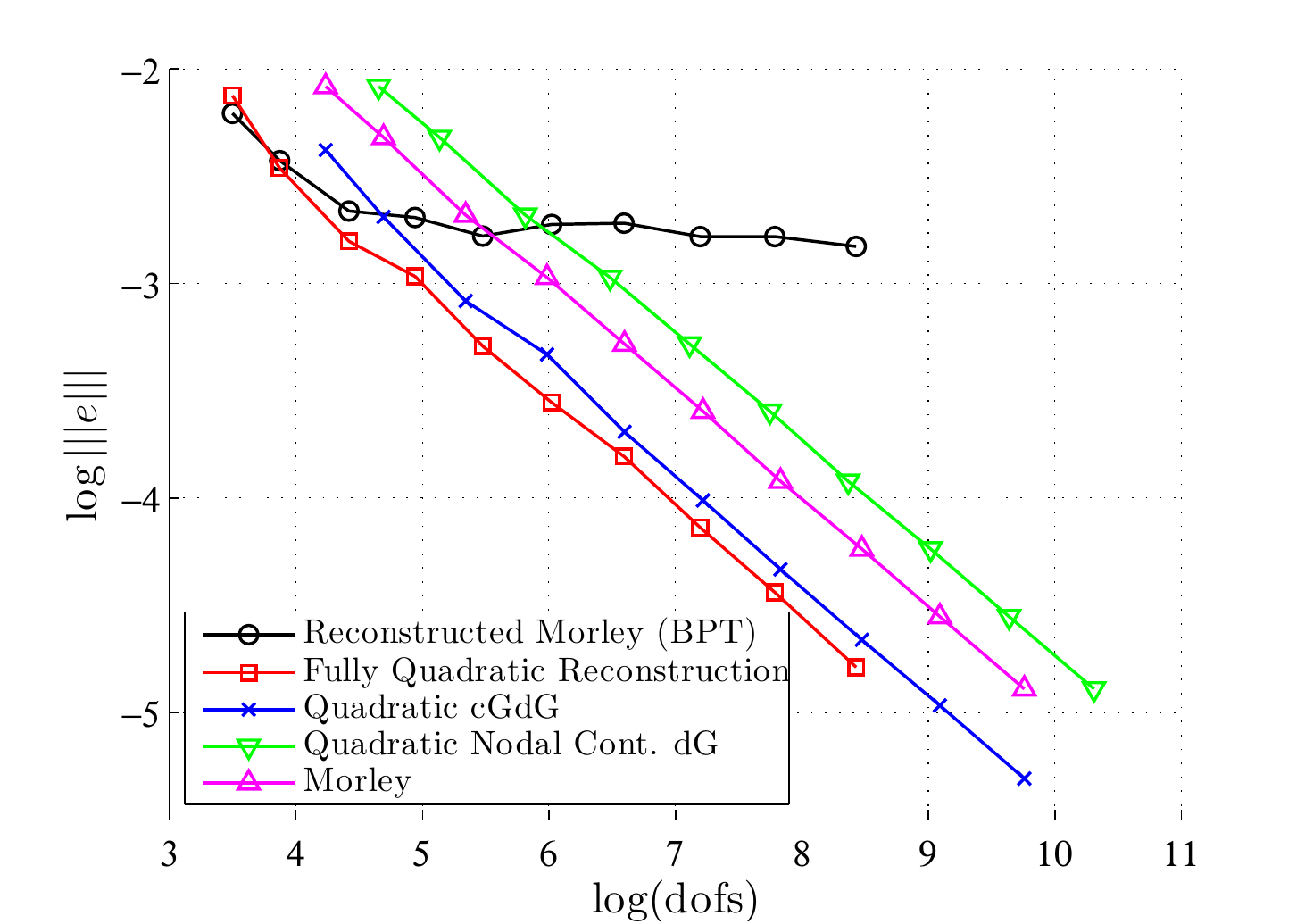}
\caption{The error in the numerical solution of Problem 2 versus the number of degrees of freedom needed. Unstructured meshes are used and the error $e$ is measured in energy norm.}
\label{fig:p2unstructuredE_ndofs}
\end{figure}

\subsubsection{Size of penalty parameter $\beta$}
In Figure~\ref{fig:p1unstructE_beta} we present some numerical results for various $\beta$. As might be suspected the fully quadratic reconstruction exhibits locking effects when $\beta$ is to large. This is natural as neighbouring elements share much of the information through the patch construction. A more surprising result is that the quadratic cG/dG method does not seem to exhibit such locking effects for large $\beta$. This indicates that the finite element space of continuous piecewise quadratic polynomials with continuous normal gradients on edge midpoints is large enough to accurately approximate the solution.
If we on the other hand change the projection operator in the penalty term from the projection onto constants $P_0$ to the projection onto linear functions $P_1$ the cG/dG method exhibits locking effects for large $\beta$.

A mesh independent lower bound for $\beta$ can be calculated if a suitable choice of $h$ in \eqref{eq:BilinearQuad} on each edge is made, see \cite{cGdGHansboLarson}. However, for the numerical results in this paper we have used a global mesh size parameter for $h$. As the meshes used in the numerical results in this paper are quasi-uniform this should be sufficient.

\begin{figure}
\centering
\includegraphics[width=10cm]{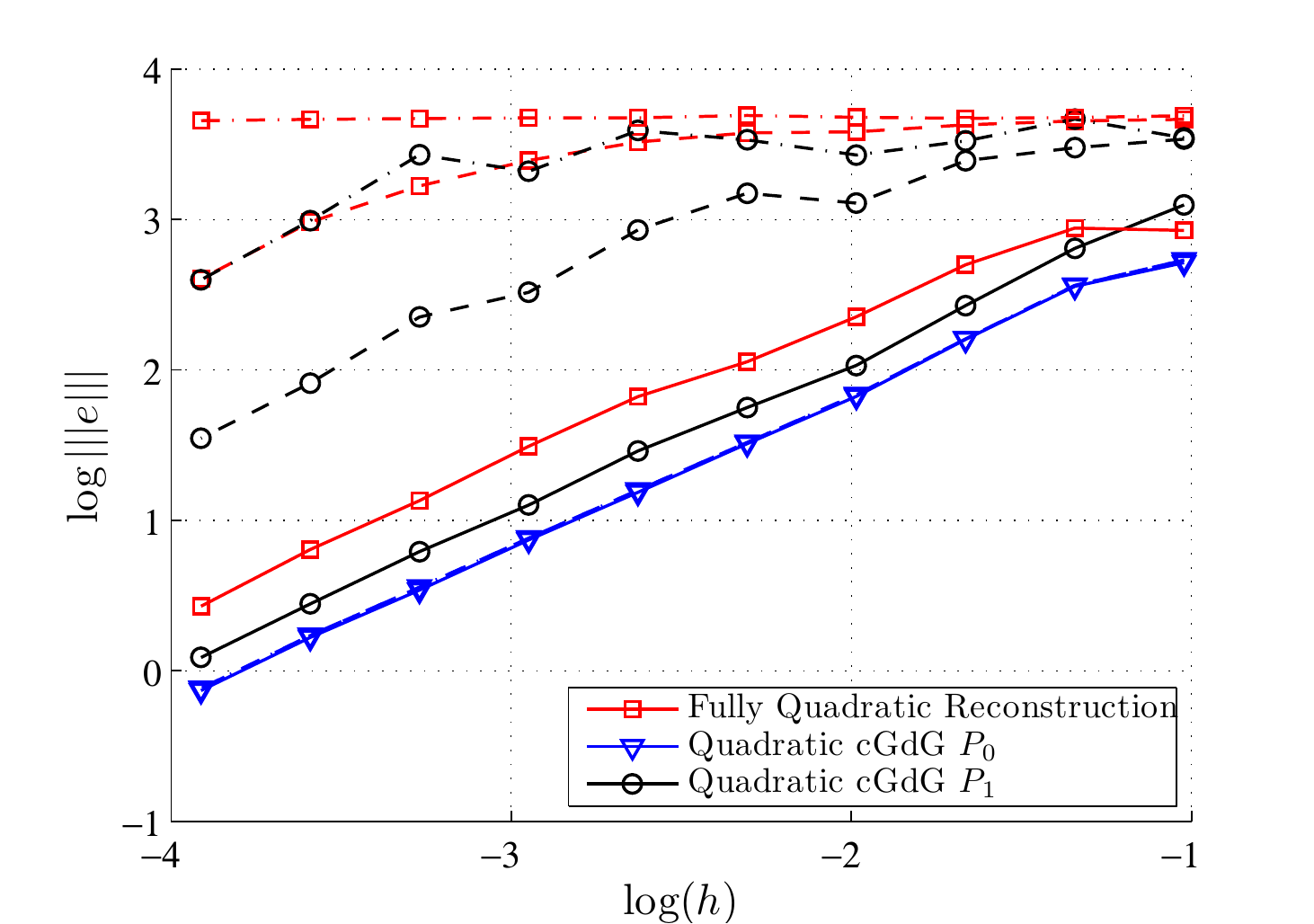}
\caption{The error in the numerical solution of Problem 1 versus the mesh size using various $\beta$. Solid lines indicate $\beta=10^2$, dashed lines indicate $\beta=10^4$, and dash-dot lines indicate $\beta=10^6$. Unstructured meshes are used and the error $e$ is measured in energy norm.}
\label{fig:p1unstructE_beta}
\end{figure}

%-----------------------------------------------------------------------------
\bibliography{PlateCP1Ref}

\begin{thebibliography}{10}
\providecommand{\url}[1]{{#1}}
\providecommand{\urlprefix}{URL }
\expandafter\ifx\csname urlstyle\endcsname\relax
  \providecommand{\doi}[1]{DOI~\discretionary{}{}{}#1}\else
  \providecommand{\doi}{DOI~\discretionary{}{}{}\begingroup
  \urlstyle{rm}\Url}\fi

\bibitem{Argyris1968}
Argyris, J., Fried, I., Scharpf, D.: The {TUBA} family of plate elements for
  the matrix displacement method.
\newblock Aeronaut. J. Roy. Aeronaut. Soc. \textbf{72}, 701--709 (1968)

\bibitem{Barnes1977}
Barnes, M.: Form finding and analysis of tension space structure by dynamic
  relaxation.
\newblock Ph.D. thesis, Dep. Civ. Engrg., The City University, London (1977)

\bibitem{BarthLarson2002}
Barth, T.J., Larson, M.G.: A posteriori error estimates for higher order
  {G}odunov finite volume methods on unstructured meshes.
\newblock In: Finite volumes for complex applications, {III} ({P}orquerolles,
  2002), pp. 27--49. Hermes Sci. Publ., Paris (2002)

\bibitem{Blum}
Blum, H., Rannacher, R.: On the boundary value problem of the biharmonic
  operator on domains with angular corners.
\newblock Math. Methods Appl. Sci. \textbf{2}, 556--581 (1980)

\bibitem{BS}
Brenner, S.C., Scott, L.R.: The mathematical theory of finite element methods,
  second edn.
\newblock Springer-Verlag (2002)

\bibitem{Engel2002}
Engel, G., Garikipati, K., Hughes, T., Larson, M., Mazzei, L., Taylor, R.:
  Continuous/discontinuous finite element approximations of fourth-order
  elliptic problems in structural and continuum mechanics with applications to
  thin beams and plates, and strain gradient elasticity.
\newblock Comput. Meth. Appl. Mech. Eng. \textbf{191}(34), 3669 -- 3750 (2002)

\bibitem{Hampshire1992}
Hampshire, J., Topping, B., Chan, H.: Three node triangular elements with one
  degree of freedom per node.
\newblock Engrg. Comput. \textbf{9}, 49--62 (1992)

\bibitem{HansboLarson2002}
Hansbo, P., Larson, M.G.: A discontinuous {G}alerkin method for the plate
  equation.
\newblock Calcolo \textbf{39}(1), 41--59 (2002)

\bibitem{cGdGHansboLarson}
Hansbo, P., Larson, M.G.: A posteriori error estimates for
  continuous/discontinuous {G}alerkin approximations of the {K}irchhoff-{L}ove
  plate.
\newblock Comput. Meth. Appl. Mech. Eng. \textbf{200}(47-48), 3289--3295 (2011)

\bibitem{Morley1968}
Morley, L.: The triangular equilibrium element in the solution of plate bending
  problems.
\newblock Aeronaut. Quart. \textbf{19}, 149--169 (1968)

\bibitem{NayUtku1972}
Nay, R., Utku, S.: An alternative to the finite element method.
\newblock In: Variational Methods in Engineering 1. University of Southampton
  (1972)

\bibitem{OnateCervera1993}
Oñate, E., Cervera, M.: Derivation of thin plate bending elements with one
  degree of freedom per node: A simple three node triangle.
\newblock Engrg. Comput. \textbf{10}, 543--561 (1993)

\bibitem{OnateZarate2000}
Oñate, E., Zárate, F.: Rotation-free triangular plate and shell elements.
\newblock Internat. J. Numer. Methods Engrg. \textbf{47}(1-3), 557--603 (2000)

\bibitem{PhaalCalladine1992}
Phaal, R., Calladine, C.: A simple class of finite elements for plate and shell
  problems. {I}: Elements for beams and thin flat plates.
\newblock Internat. J. Numer. Methods Engrg. \textbf{35}, 955--977 (1992)

\bibitem{PhaalCalladine1992-2}
Phaal, R., Calladine, C.: A simple class of finite elements for plate and shell
  problems. {II}: An element for thin shells, with only translational degrees
  of freedom.
\newblock Internat. J. Numer. Methods Engrg. \textbf{35}, 979--996 (1992)

\bibitem{Timoshenko}
Timoshenko, S., Woinowsky-Krieger, S.: Theory of plates and shells, second edn.
\newblock McGraw-Hill (1959)

\end{thebibliography}
\bibliographystyle{spmpsci}
%-----------------------------------------------------------------------------
\end{document}